\documentclass[12pt]{amsart}
\usepackage{amsfonts}
\usepackage{amsmath}
\usepackage{amssymb}
\usepackage[margin=1.2in]{geometry}
\usepackage{verbatim}
\usepackage{amsthm}
\usepackage{stmaryrd}

\usepackage{paralist}

\usepackage{xcolor}
\usepackage[hidelinks]{hyperref}
\hypersetup{
    colorlinks = true,
    linkcolor = red,
    anchorcolor = red,
    citecolor = green,
    filecolor = red,
    urlcolor = red,
    linktoc=page
    }
\usepackage{cleveref}
\usepackage{tikz}
\usetikzlibrary{shapes.geometric,decorations.pathmorphing,patterns}
\usepackage[font=small,labelfont=bf]{caption}
\usepackage{graphicx}
\graphicspath{{./figures/induction_heating/}}
\usepackage{subcaption}
\usepackage{comment}

\makeatletter
\def\namedlabel#1#2{\begingroup
   \def\@currentlabel{#2}%
   \label{#1}\endgroup
}
\makeatother

\newcommand{\R}{\mathbb R}

\newcommand{\N}{\mathbb N}
\newcommand{\Z}{\mathbb Z}

\newcommand{\RR}{\mathcal R}

\newcommand{\TT}{\mathcal T}
\newcommand{\OO}{\mathcal O}

\newcommand{\transpose}{\mathsf{T}}

\DeclareMathOperator{\Ls}{L}
\DeclareMathOperator{\LLs}{\mathbf L}
\DeclareMathOperator{\Hs}{H}
\DeclareMathOperator{\HHs}{\mathbf H}
\DeclareMathOperator{\Ws}{W}

\DeclareMathOperator{\Cs}{C}
\DeclareMathOperator{\ACs}{AC}
\DeclareMathOperator{\CCs}{\mathbf C}
\DeclareMathOperator{\Lips}{Lip}

\DeclareMathOperator{\Xs}{X}
\DeclareMathOperator{\Ys}{Y}
\DeclareMathOperator{\Zs}{Z}

\DeclareMathOperator*{\essup}{ess\,sup}

\newcommand{\brac}[1]{\left\lbrace{#1}\right\rbrace}

\newcommand{\paren}[1]{\left({#1}\right)}
\newcommand{\norm}[1]{\left\lVert{#1}\right\rVert}

\newcommand{\abs}[1]{\left\vert{#1}\right\vert}

\newcommand{\inprod}[1]{\left\langle{#1}\right\rangle}

\newcommand{\jump}[1]{\llbracket{#1}\rrbracket}

\newcommand{\ovl}[1]{\overline{#1}}
\newcommand{\unl}[1]{\underline{#1}}

\newcommand{\bm}[1]{\boldsymbol{#1}}
\newcommand{\dive}[1]{\nabla\cdot{#1}}
\newcommand{\curl}[1]{\nabla\times{#1}}

\newcommand{\veps}{\varepsilon}

\newcommand{\pa}{\partial}

\newcommand{\Gm}{\Gamma}
\newcommand{\om}{\omega}
\newcommand{\Om}{\Omega}
\newcommand{\sm}{\sigma}

\newcommand{\Sm}{\Sigma}

\newcommand{\fa}{\forall}
\newcommand{\sst}{\subset}

\newcommand{\Hb}{\bm{H}}
\newcommand{\Bb}{\bm{B}}
\newcommand{\Eb}{\bm{E}}

\newcommand{\Jb}{\bm{J}}
\newcommand{\Ab}{\bm{A}}

\newcommand{\Wb}{\bm{W}}
\newcommand{\xb}{\bm{x}}

\newcommand{\vb}{\mathbf{v}}

\newcommand{\fb}{\bm{f}}

\newcommand{\qb}{\bm{q}}

\newcommand{\pb}{\bm{p}}

\newcommand{\nv}{\textbf{n}}

\newcommand{\zrb}{\bm{0}}
\newcommand{\vphib}{\bm{\varphi}}

\newcommand{\Phib}{\bm{\Phi}}

\newcommand{\alb}{\bm{\alpha}}

\newcommand{\sra}{\rightharpoonup}
\newcommand{\emb}{\hookrightarrow}

\newcommand{\leqc}{\lesssim}
\newcommand{\geqc}{\gtrsim}

\newcommand{\dx}{\,\mathrm{d}\xb}
\newcommand{\ds}{\,\mathrm{d}s}
\newcommand{\dt}{\,\mathrm{d}t}
\newcommand{\di}{\,\mathrm{d}}

\newcommand{\q}{\quad}
\newcommand{\qq}{\qquad}
\newcommand{\qqq}{\qquad\quad}
\newcommand{\qqqq}{\qquad\qquad}

\newcommand{\qqqqqq}{\qquad\qquad\qquad}
\newcommand{\qqqqqqq}{\qquad\qquad\qquad\quad}

\newtheorem{theorem}{Theorem}[section]
\newtheorem{lemma}{Lemma}[section]

\newtheorem{definition}{Definition}[section]

\newtheorem{remark}{Remark}[section]

\makeatletter
\@namedef{subjclassname@2020}{%
  \textup{2020} Mathematics Subject Classification}
\makeatother

\begin{document}

	\title[]{A numerical scheme for solving an induction heating problem with moving non-magnetic conductor}
	
	\author[V.~C. Le]{Van Chien Le$^1$}
\thanks{The work of V.~C.~Le was supported by the European Research Council through the European Union's Horizon 2020 Research and Innovation programme (Grant number 101001847).}

\author[M. Slodi\v{c}ka]{Mari\'{a}n Slodi\v{c}ka$^2$}

	\author[K. Van Bockstal]{Karel Van Bockstal$^3$} 
	\thanks{The work of K.~Van Bockstal was supported by the Methusalem programme of Ghent University Special Research Fund (BOF) (Grant Number 01M01021).} 
	
\address[1]{IDLab, Department of Information Technology, Ghent University - imec, B 9000 Ghent, Belgium}

\email{vanchien.le@ugent.be}

\address[2]{Research group NaM$^2$, Department of Electronics and Information Systems, Ghent University, B 9000 Ghent, Belgium}
\email{marian.slodicka@ugent.be}
 
	\address[3]{Ghent Analysis \& PDE center, Department of Mathematics: Analysis, Logic and Discrete Mathematics\\ Ghent University\\
		Krijgslaan 281\\ B 9000 Ghent\\ Belgium} 
	\email{karel.vanbockstal@UGent.be}

	\subjclass[2020]{35Q61, 35Q79, 65M12}
	\keywords{induction heating, multi-component system, moving non-magnetic conductor, Reynolds transport theorem, restrained Joule heat source}
	
	\begin{abstract} 
This paper investigates an induction heating problem in a multi-component system containing a moving non-magnetic conductor. The electromagnetic process is described by the eddy current model, and the heat transfer process is governed by the convection-diffusion equation. The two processes are coupled by a restrained Joule heat source. A temporal discretization scheme is introduced to numerically solve the corresponding variational system. With the aid of the Reynolds transport theorem and a density argument, we prove the convergence of the proposed scheme as well as the well-posedness of the variational problem. Some numerical experiments are also performed to assess the performance of the numerical scheme.
	\end{abstract}
	
	\maketitle
	
	\tableofcontents

\section{Introduction}
\label{sec:intro}

Induction heating is the process of heating an electrical conductor through the heat generated by an eddy current. Induction heating is a standard industrial process with various applications, including surface hardening, induction mass heating, induction melting and induction welding. Other industrial applications of induction heating are listed in \cite{Rudnev2017}. Basically, an alternating current is passed through an electric coil. The electromagnetic fields occurring in the surrounding space induce an electric current in electrically conductive mediums, which is called the eddy current. Then, heat is generated due to the resistance of materials to the eddy current or the so-called Joule heating effect.

A considerable amount of literature has been published on the study of the induction heating process. The majority stand on physical and engineering points of view, where numerical simulation strategies have been performed, and experiments have been set up to validate numerical results. As an instance, the modelling of induction heating of carbon steel tubes was carried out in \cite{LFA2012}. The authors considered a mathematical model combining electromagnetic process, heat transfer by conduction, convection, and radiation, and ferromagnetic-paramagnetic transition. Some numerical simulations of the heating stage were made using the finite-element method (FEM) and were validated by experimental measurements. We also refer the reader to \cite{ZA2012,LWSG2020} for other studies of stationary induction heat treating. Besides that, FEM-based numerical schemes for moving induction heating problems were also studied in \cite{WHJ+1999,SG2012,BZW2015,WH2017}. However, those papers did not investigate fundamental questions, such as the convergence and stability of numerical simulations and the properties of the solution. 

In contrast to the aforementioned papers, several studies to date have investigated the well-posedness and the regularity of the solution to induction heating problems. However, all of them were restricted to a static geometry. The authors of \cite{Yin1994,Yin1997,Yin1998,Bien1999} studied the global solvability of Maxwell's equations together with temperature effects. More specifically, in \cite{Yin1994}, the quasi-static Maxwell's equations were expressed in terms of the magnetic field, and the existence of a solution was proved using a fixed point argument. The regularity of the solution was then studied in \cite{Yin1997}. In \cite{Yin1998}, the existence of a solution was shown for Maxwell's equations with the electric and magnetic fields as unknowns. The author of \cite{Bien1999} considered a degenerate problem modelling Joule heating in a conductive medium. The existence of global-in-time weak solutions was proved via the Faedo-Galerkin method. The papers \cite{SC2017,CS2017,CGS2017} also concerned mathematical models for a stationary induction heating problem. Herein, the electromagnetic process and heat transfer are both governed by nonlinear equations. In \cite{SC2017}, the equation was derived from Maxwell's equations in terms of the magnetic field, whilst in \cite{CS2017} it was expressed in terms of the magnetic induction. In both articles, the authors proved the existence of a weak solution to the coupled system with controlled Joule heating. The problem was then formulated in terms of the magnetic vector potential and electric scalar potential ($\Ab-\phi$ formulation) in \cite{CGS2017}. The existence of a global solution to the whole system was shown, and a numerical simulation was performed to support theoretical results. 

Recently, some theoretical and numerical studies on moving electromagnetic problems have been published. In \cite{BRR+2011,BMR+2013,BLR+2019}, the authors considered an eddy current problem in a cylindrical symmetric domain containing a moving non-magnetic conductor. The well-posedness of the variational system was studied, and a numerical scheme was introduced for the computation of solution. These results were extended to a general three-dimensional domain (without the symmetry assumption) in \cite{LSV2021b} and \cite{LSV2022c}. In these papers, a temporal discretization based on the backward Euler method and a FEM-based space-time discretization scheme were respectively proposed. The corresponding error estimates were also established, and some numerical experiments were introduced to validate the performance of the proposed schemes.  
In addition to those papers, the authors of \cite{LSV2021a} considered an electromagnetic contact problem with a moving conductor. The restriction on a non-magnetic moving conductor was no longer made. Instead, it allowed material coefficients to be fully jumping. In this case, the well-posedness of the system was proved using Rothe's method. These pioneering works on moving electromagnetic problems serve as a basis for the mathematical analysis and numerical computation of the induction heating process involving moving conductors.

To the best of our knowledge, there has been no paper dealing with mathematical analysis of an induction heating problem with a moving conductor, even though this process has successfully been applied in industry for decades. The present paper investigates an induction heating problem in a multi-component system containing a moving non-magnetic conductor. 
The electromagnetic process is described by the eddy current model, which is coupled with heat transfer via the Joule heating effect. Due to the conductor's and surrounding air's movement, the heat transfer process is a combination of thermal conduction and convection mechanisms. The nonlinearity of the Joule heat source is treated by introducing a cut-off function. Our investigation also relies on the assumption that the moving conductor is filled by a non-magnetic material. 

This paper is organised into eight sections. The next one introduces the geometrical setting and describes a mathematical model of the induction heating process. \Cref{sec:functional_setting} provides function spaces and some auxiliary results, which are necessary for analysis of the governed problem. \Cref{sec:uniqueness} derives the corresponding variational system and shows the uniqueness of its solution. In \Cref{sec:TD}, we design a temporal discretization scheme based on the backward Euler method and perform some a priori estimates for iterates. \Cref{sec:existence} is the central section devoted to showing the existence of a solution to the variational system and the convergence of the proposed numerical scheme. Finally, we present some numerical results for the discretization scheme in \Cref{sec:numerial_results}, and then we give a conclusion and some possibilities for future work in \Cref{sec:conclusion}.

\section{Mathematical model} 
\label{sec:model}

\subsection{Geometrical setting}

We adopt the geometrical setting described in \cite{LSV2021a}, which was introduced for a moving electromagnetic problem. Let $\Om$ be an open, simply-connected, and bounded domain in $\R^3$ such that its boundary $\pa \Om$ belongs to the class $\Cs^{1,1}$ or $\Om$ is a convex polyhedron. The domain $\Om$ contains a moving workpiece $\Sm$ and a fixed coil $\Pi$ that are surrounded by air. The open connected subdomains $\Sm$ and $\Pi$ are supposed to belong to the class $\Cs^{2,1}$ and separate from each other, see \Cref{fig:domain}. Moreover, we introduce some notations that are frequently used throughout the manuscript: $\nv$ denotes the outward unit normal vector to the boundary $\pa\Om$; $\Theta(t) := \Sm(t) \cup \Pi$ is the subdomain consisting of electrically conductive bodies at time $t$; $\Xi(t) := \Om \setminus \ovl{\Theta(t)}$ is the space occupied by air at time $t$; and the interval $[0, T]$ stands for the considered time frame. The coil $\Pi$ shares common interfaces with the boundary $\pa \Om$, denoted by $\Gamma := \Gm_{\textrm{in}} \cup \Gm_{\textrm{out}}$, whose measures are strictly positive, i.e., $|\Gm_{\textrm{in}}|>0$ and $|\Gm_{\textrm{out}}| > 0$.

\begin{figure}
	\centering
	\begin{tikzpicture}
		\node [cylinder, shape border rotate=0, draw, minimum height=1.6cm, minimum width=0.8cm, fill=white] {};
		\draw[decoration={aspect=0.2,segment length=1.2mm,amplitude=0.65cm,coil},decorate] (2,0) -- (4,0); 
		\draw (2, 0) -- (2, -2.95);
		\draw (4, 0) -- (4, -1.67);
		\draw (1.5,0) circle (3cm);
		\draw (-1.5,0) arc (180:360:3 and 1.2);
		\draw[dashed] (4.5,0) arc (0:180:3 and 1.2);
		\draw[->, line width=0.2mm] (0.75,0) -- (1.3,0);
		\node at (0,0) {$\Sm$};
		\node at (1.15, 0.15) {$\vb$};
		\node at (4.5, 2) {$\Om$};
		\node at (2.5, 0.85) {$\Pi$};
		\node at (1.5, 2) {$\Xi$};
		\node at (2.1, -3.25) {$\Gm_{\textrm{in}}$};
		\node at (4.25, -1.95) {$\Gm_{\textrm{out}}$};
	\end{tikzpicture}
	\caption{The domain $\Omega$ consisting of a workpiece $\Sm$ moving with velocity $\vb$, a fixed coil $\Pi$ and the surrounding air $\Xi$. The coil $\Pi$ shares common interfaces $\Gm_{\textrm{in}}$ and $\Gm_{\textrm{out}}$ of strictly positive measures with the boundary (see \cite[Figure~1]{LSV2021b}).}
	\label{fig:domain}
\end{figure}
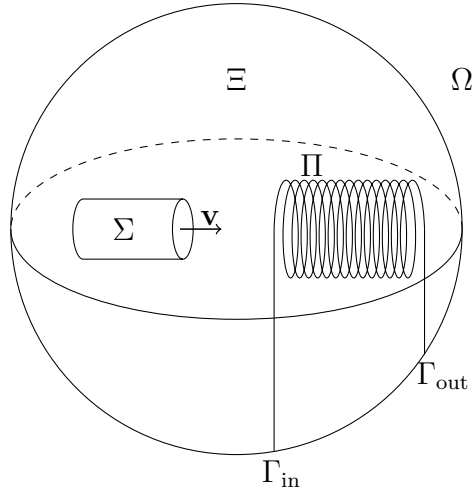
The movement of the workpiece can be parameterized by a smooth bijective mapping $\Phib: \Sm(0) \times [0, T] \to \R^3$ such that
\begin{equation}
	\label{eq:uniform_map}
	\Sm(t) = \Phib(\Sm(0), t);
	\q \bigcup_{t\in [0,T]} \ovl{\Sm(t)} \subset \Om;
	\q \det \nabla \Phib(\xb, t) > 0, \q \fa (\xb, t) \in \Sm(0) \times [0, T].
\end{equation}
We define the trajectory of the motion
\[
\TT := \brac{(\xb, t) : \xb \in \Sm(t), \ t \in [0, T]},
\]
and the velocity $\vb: \TT \to \mathbb{R}^3$ of the workpiece 
\[
\vb(\xb, t) = \dot{\Phib}\left(\left[\Phib(\cdot, t)\right]^{-1}(\xb), t\right),
\]
where $\dot{\Phib}$ represents the total derivative of $\Phib$ with respect to (w.r.t.) time variable. The reader is referred to \cite[Section~8]{Gurtin1981} for more details. Due to the movement of the workpiece $\Sm$, the surrounding air also moves in a fluid manner. We assume that the velocity vector can be extended to the whole domain $\Om$ and this extension (also denoted by $\vb$) is of class $\CCs^1(\ovl{\Om} \times [0, T])$ and satisfies $\vb = \zrb$ in the coil $\ovl{\Pi}$. 

The subdomains $\Sm, \Pi$ and $\Xi$ are filled by different materials, e.g., aluminium workpiece, copper coil and air. Induction heating process involves the following material coefficients: the magnetic permeability $\mu$, the electrical conductivity $\sm$, the thermal conductivity $\kappa$, and the volumetric heat capacity $\alpha$. In the case of non-magnetic conductors, the magnetic permeability on the entire domain can be well approximated by the constant of vacuum $\mu_0 > 0$, i.e.,
\[
\mu = \mu_0 \qq \text{ in } \quad \Omega.
\]
For the sake of simplicity, we assume that all material coefficients are strictly positive constants on each subdomain, except the electrical conductivity $\sm$ that is vanishing on the air, i.e.,
\[
\sigma(t) = 
\begin{cases}
	\sigma_{\Pi} > 0  \qq & \text{in} \q  \Pi, \\
	\sigma_{\Sigma} > 0 \qq & \text{in} \q \Sigma(t), \\
	0 & \text{in} \q \Xi(t). 
\end{cases}
\]

\begin{remark}
	Please note that the material coefficients $\sm, \kappa$ and $\alpha$ (except the magnetic permeability $\mu$) might be discontinuous at the interface of different subdomains. We use the subscripts $\Sm, \Pi$ and $\Xi$ to distinguish the material functions on the workpiece, the coil and air, respectively. 
\end{remark}

\subsection{Mathematical model}
Mathematical modelling of a low-frequency electromagnetic system with a moving conductor was thoroughly discussed in \cite{LSV2021b,LSV2021a}. Let us briefly recall the model considered in these papers. The electromagnetic process is modelled by the eddy current approximation of Maxwell's equations, or the so-called quasi-static system
\begin{subequations}
	\begin{align}
		& \dive{\Bb} = 0, \label{eq:Gauss} \\
		& \curl{\Eb} = -\pa_t\Bb, \label{eq:Faraday} \\
		& \curl{\Hb} = \Jb, \label{eq:Ampere}
	\end{align}
\end{subequations}
where $\Eb, \Hb, \Bb$ and $\Jb$ stand for the electric field, magnetic field, magnetic induction and current density, respectively. The behaviour of the electromagnetic fields passing through the interface of different materials is expressed by the following transmission conditions
\begin{equation}\label{eq:interface_condition}
	\llbracket\Bb \cdot \nv\rrbracket_{\pa \Theta \setminus \Gamma} = 0, \qq \llbracket\Hb \times \nv\rrbracket_{\pa \Theta \setminus \Gamma} = \zrb, \qq \text{and} \qq \llbracket (\Eb + \vb \times \Bb) \times \nv\rrbracket_{\pa \Theta \setminus \Gamma} = \zrb,
\end{equation}
where the unit normal vector $\nv$ to $\pa\Theta \setminus \Gm$ points from the electrical conductors (i.e., the workpiece $\Sm$ and the coil $\Pi$) to the air, and the jumps are defined by
\[
\llbracket \fb \times \nv\rrbracket = (\fb_2 - \fb_1) \times \nv, \qqqq \llbracket \fb \cdot \nv\rrbracket = (\fb_2 - \fb_1) \cdot \nv,
\] 
with $\fb_1$ and $\fb_2$ the limiting values of the field $\fb$ from the conductors and the air, respectively. 
We introduce a vector potential $\Ab$ of the magnetic induction $\Bb$ such that $\Bb = \curl{\Ab}$. When $\Bb \cdot \nv = 0$ on the boundary $\pa\Om$, the vector potential $\Ab$ exists uniquely 
such that $\Ab$ is divergence-free and satisfies $\Ab \times \nv  = \zrb$ on $\pa \Om$ (cf. \cite[Theorem 3.6 on p. 48]{Girault1986}).
Substituting $\Bb = \curl{\Ab}$ into the Faraday law \eqref{eq:Faraday} leads us to the decomposition of the electric field $\Eb = -\pa_t \Ab - \nabla \phi$, where $\phi$ exists uniquely up to a constant.
In addition, the general Ohm's law provides a constitutive relation for Maxwell's equations
\[
\Jb = \sm (\Eb + \vb \times \Bb).
\]
Hence, the total current density $\Jb$ can be divided into a source current part $\Jb_s = -\sm \nabla \phi$ and an eddy current part $\Jb_e = -\sm \pa_t \Ab + \sm \vb \times (\curl{\Ab})$. 
The source current $\Jb_s$ is originated from an external current $j$ applied on the interfaces $\Gm_{\textrm{in}}$ and $\Gm_{\textrm{out}}$. The scalar potential $\phi$ on the coil $\Pi$ is the solution to the following boundary value problem \cite{Homberg2004}
\begin{equation}
	\label{eq:phiBVP}
	\begin{cases}
		\dive{\paren{-\sm\nabla\phi}} = 0 \qq & \text{in} \q \Pi \times (0, T),\\
		-\sm\nabla\phi\cdot\nv = 0  & \text{on} \q \paren{\pa\Pi\setminus\Gm} \times (0, T),\\
		-\sm\nabla\phi\cdot\nv = j & \text{on} \q \Gm \times \paren{0, T},
	\end{cases}
\end{equation}
where $j$ satisfies the following compatibility condition 
\begin{equation}\label{eq:compatibility}
	\int\limits_{\Gm} j(\xb, t) \ds = 0 \qqqq \fa t \in [0, T].
\end{equation}
A comprehensive explanation of the modelling of the source current on the workpiece and on the air can be found on \cite[p.~4]{LSV2021a}. Finally, thanks to the Amp\`{e}re relation \eqref{eq:Ampere}, the initial-boundary value problem of the vector potential $\Ab$ reads as
\begin{equation}
	\label{eq:AIBVP}
	\begin{cases}
		\sm \pa_t\Ab + \mu_0^{-1} \curl{\curl{\Ab}} + \chi_{\Pi}\sm\nabla \phi - \sm \vb \times (\curl{\Ab}) = \zrb \qq & \text{in }  \Om \times (0, T), \\
		\dive{\Ab} = 0 & \text{in }  \Om \times (0, T), \\
		\Ab \times \nv = \zrb & \text{on }  \pa\Om \times (0, T), \\
		{\llbracket \Ab \rrbracket = \zrb} & {\text{on }  (\pa \Theta \setminus \Gamma) \times (0, T),} \\
		\llbracket (\curl{\Ab}) \times \nv \rrbracket = \zrb & \text{on } (\pa \Theta \setminus \Gamma) \times (0, T), \\
		\Ab(\cdot, 0) = \tilde{\Ab}_0 &  \text{in }  \Theta(0),
	\end{cases}
\end{equation}
where $\chi_{\Pi}$ is the characteristic function of the domain $\Pi$. By the Joule heating effect, the electric current flowing through the conductors produces a significant amount of heat given by
\[
Q = \dfrac{1}{\sm}\abs{\Jb}^2 = \sm \abs{\pa_t\Ab + \chi_{\Pi} \nabla\phi - \vb \times \left(\curl{\Ab}\right)}^2.
\]
This Joule heat source plays the role of an internal (contactless) source, which represents the coupling of the electromagnetic process and heat transfer. It is one of the most challenging points during the mathematical treatment of the model. In order to restrain this quadratic source term from increasing uncontrollably, we introduce a cut-off function $\RR_r$ that truncates the source heat $Q$ by a constant $r > 0$ as follows
\[
\RR_r(Q)(\xb, t) = \min(r, Q(\xb, t)).
\]
From the engineering point of view, this truncation models the use of a switch-off button, which prevents the conductors from undesirable thermal deformations. Due to the movement of the workpiece and surrounding air, the heat transfer process is governed by thermal conduction and thermal convection, which are described by the convection-diffusion equation. On the boundary $\pa \Om$, we impose a homogeneous Neumann condition, {which describes the entire domain $\Om$ as an isolated system}. Hence, the temperature $u$ is the solution to the following initial-boundary value problem (see, e.g., \cite{Bejan2013})
\begin{equation}
	\label{eq:uIBVP}
	\begin{cases}
		\alpha \pa_t u + \alpha \vb \cdot \nabla u - \dive{\left(\kappa \nabla u\right)} = \RR_r(Q) \qqq & \text{in}  \q \Om \times \paren{0, T}, \\
		\kappa \nabla u \cdot \nv = 0 & \text{on} \q \pa \Om \times \paren{0, T}, \\
		u(\cdot, 0) = \tilde{u}_0 & \text{in} \q \Om.
	\end{cases}
\end{equation}
The following transmission conditions describe the perfect thermal contact (without friction) between the conductors and the environment
\begin{equation}
	\label{eq:interface_u}
	\jump{u}_{\pa\Theta \setminus \Gm} = 0, \qqqqqqq \jump{\kappa \nabla u \cdot \nv}_{\pa\Theta \setminus \Gm} = 0.
\end{equation}

\begin{remark}
	The problem \eqref{eq:uIBVP}-\eqref{eq:interface_u} can be reformulated by substituting $\tilde{u} := \alpha u$. The resulting formulation has been studied in \cite{Lehrenfeld2015} and references therein. That formulation gets rid of the jumping coefficient $\alpha$ associated with the time derivative $\pa_t u$, but it in turn gives rise to the discontinuity of the solution $\tilde{u}$ across the interface $\pa\Theta(t) \setminus \Gm$. In this paper, we take advantage of classical Sobolev spaces by considering the system \eqref{eq:uIBVP}-\eqref{eq:interface_u}. 
\end{remark}

\section{Functional setting}
\label{sec:functional_setting}

\subsection{Standard function spaces}
First of all, we introduce some function spaces that are frequently used. 
The Sobolev space $\Ws^{k, \lambda}(\Om)$, with $k \in \N$ and $\lambda \in [1, \infty)$, is equipped with the following norm
\[
\norm{f}_{\Ws^{k, \lambda}(\Om)} = \paren{\sum\limits_{0 \le \abs{\alb} \le k} \int\limits_{\Om} \abs{D^{\alb} f(\xb)}^\lambda \dx}^{1/\lambda}.
\]
When $k = 0$, the space $\Ws^{0, \lambda}(\Om)$, with $\lambda \in [1, \infty)$, becomes the Lebesgue space $\Ls^{\lambda}(\Om)$. We denote by $\paren{\cdot, \cdot}_{\Om}$ the scalar product of the space $\Ls^2(\Om)$, with its induced norm $\norm{\cdot}_{\Ls^2(\Om)}$.
Among Sobolev spaces, only $\Ws^{k, 2}(\Om)$, with $k \in \N$, forms a Hilbert space, which is denoted by $\Hs^k(\Om)$. The notation $\Hs^1_0(\Om)$ stands for the closure of $\Cs^{\infty}_0(\Om)$ w.r.t. the norm of $\Hs^1(\Om)$, where $\Cs^{\infty}_0(\Om)$ is the space of compactly supported smooth functions defined on $\Om$. The trace space $\Hs^{1/2}(\pa\Om)$ consists of the trace of functions in $\Hs^1(\Om)$ to the boundary $\pa\Om$. The dual space of $\Hs^1(\Om)$ and $\Hs^{1/2}(\pa\Om)$ are denoted by $[\Hs^1(\Om)]^\prime$ and $\Hs^{-1/2}(\pa\Om)$, whose elements are identified via the $\Ls^2(\Om)$ and $\Ls^2(\Gm)$ pairings, respectively. Their duality pairings are respectively denoted by $\inprod{\cdot, \cdot}_{1, \Om}$ and $\inprod{\cdot, \cdot}_{1/2, \Gm}$. These notations are inherited for vector and tensor fields by using the corresponding bold symbols.
Moreover, the subspace $\Zs$ of $\Hs^1(\Pi)$ defined by
\[
\Zs := \brac{f \in \Hs^1(\Pi) : \paren{f, 1}_{\Pi} = 0}
\]
is a Hilbert space with the equivalent norm $\norm{\nabla f}_{\LLs^2(\Om)}$. In addition, the following Banach space of vector fields plays a central role in further analysis
\[
\Wb_0 := \brac{\fb \in \LLs^2(\Om) : \curl{\fb} \in \LLs^2(\Om), \dive{\fb} = 0, {\left.\fb\right|_{\pa\Om} \times \nv = \zrb}}
\]
equipped with the norm
\[
\norm{\fb}_{\Wb_0} = \norm{\curl{\fb}}_{\LLs^2(\Om)}.
\]
This norm is equivalent to the graph norm on the space $\Wb_0$, and $\Wb_0$ is continuously embedded into $\HHs^1(\Om)$ since the open, bounded, simply-connected domain $\Om$ is a convex polyhedron or its boundary $\pa\Om$ belongs to the class $\Cs^{1,1}$, see \cite[Lemma 3.3 on p. 51]{Girault1986} and \cite[Theorem 3.7 on p. 52]{Girault1986}. 

Next, let $\Xs$ be an arbitrary Banach space with norm $\norm{\cdot}_{\Xs}$ and $f: (0, T) \to \Xs$ be an abstract function. We denote by $\Cs([0, T], \Xs)$ and $\Lips([0, T], \Xs)$ the spaces of continuous and Lipschitz continuous functions $f$ endowed with the usual norm
\[
\norm{f}_{\Cs([0, T], \Xs)} = \max\limits_{0 \le t \le T} \norm{f(t)}_{\Xs}.
\]
The Bochner spaces $\Ls^{\lambda}((0, T), \Xs)$, with $\lambda \in [1, \infty)$, and $\Ls^{\infty}((0, T), \Xs)$ consist of all measurable abstract functions $f$ furnished with the norms
\[
\norm{f}_{\Ls^{\lambda}((0, T), \Xs)} = \paren{\int\limits_0^T \norm{f(t)}^{\lambda}_{\Xs} \dt}^{1/{\lambda}}, \qqqq \norm{f}_{\Ls^{\infty}((0, T), \Xs)} = \essup\limits_{t \in (0, T)} \norm{f(t)}_{\Xs}.
\]
Throughout this article, we denote by $\veps, C_{\veps}$ and $C$ positive constants depending only on the given data, where $\veps$ is a small number and $C_{\veps}$ is a large number depending on $\veps$. Their different values at different contexts are allowed. In order to reduce the number of constant notations, the notation $a \le Cb$ ($a \ge Cb$, resp.) is replaced by $a \leqc b$ ($a \geqc b$, resp.).

\subsection{Auxiliary results}

Since the Reynolds transport theorem (RTT) is crucial for further analysis of PDEs with moving domains, it is recalled here with some related inequalities.
For the sake of brevity, in this section, we omit the dependency of functions on $\xb$ and $t$ if it does not lead to any confusion.
We consider a Lipschitz moving domain $\om(t)$ whose movement is associated with a velocity vector $\vb$ of class $\CCs^1$. Let $f(\xb, t)$ be a scalar abstract function satisfying $f(t) \in \Ws^{1, 1}(\om(t))$ and $\pa_t f(t) \in \Ls^1(\om(t))$ for a.a. $t \in (0, T)$. Then, the RTT (cf. \cite[p.~78]{Gurtin1981}) and the Divergence theorem say
\begin{align}
	\dfrac{\mbox{d}}{\dt} \int\limits_{\om(t)} f \dx 
	& = \int\limits_{\om(t)} \pa_t f \dx + \int\limits_{\pa\om(t)} f(\vb \cdot \nv) \ds  \label{eq:Reynolds1} \\
	& = \int\limits_{\om(t)} \pa_t f \dx + \int\limits_{\om(t)} \dive{(f \vb)} \dx,\label{eq:Reynolds2}
\end{align}
where the unit normal $\nv$ to $\pa\om(t)$ points outward.
Given $f(t) \in \Hs^1(\om(t))$ for some $t$, the Divergence theorem and the $\veps$-Young inequality give that (see also \cite[Lemma~2.1]{Slodicka2021})
\begin{equation} 
	\int\limits_{\pa\om(t)} f^2 (\vb \cdot \nv) \ds 
	= 2 \int\limits_{\om(t)} f (\nabla f \cdot \vb) \dx + \int\limits_{\om(t)} f^2 (\dive{\vb}) \dx \le \veps \norm{\nabla f}^2_{\LLs^2(\om(t))} + C_{\veps} \norm{f}^2_{\Ls^2(\om(t))}. \label{eq:Necas_like}
\end{equation}
The constants $\veps$ and $C_{\veps}$ only depend on the norm of the velocity $\vb$, and the inequality \eqref{eq:Necas_like} is still valid for vector functions $\fb(t) \in \HHs^1(\om(t))$.

Next, let $\alpha \ge 0$ be a coefficient function that is constant on each subdomain $\Sm(t), \Xi(t)$ and $\Pi$, $w\in \Hs^1(\Om)$ and $u \in \Ls^1((0, T), \Hs^1(\Om))$. We denote 
\begin{equation}
	\label{def:f}
	{\mathcal F}_w(t) := (\alpha(t) u(t), w)_\Om = \int\limits_\Omega \alpha(t) u(t) \, w \dx.
\end{equation}
If ${\mathcal F}_w$ is absolutely continuous on $[0, T]$, i.e., ${\mathcal F}_w \in \ACs([0, T]),$ then its derivative ${\mathcal F}_w^\prime$ exists a.e. and belongs to the space $\Ls^1((0,T))$. In addition, if $\pa_t u(t) \in \Ls^2(\Om)$ for a.a. $t \in (0, T)$, then by the RTT we have that 
\begin{align*}
	{\mathcal F}_w^\prime (t) 
	& = \dfrac{\di}{\dt} \paren{\alpha(t) u(t), w}_\Om \\
	& = \alpha_\Sigma \dfrac{\di}{\dt} \int\limits_{\Sigma(t)} u(t) \, w \dx + \alpha_{\Xi} \dfrac{\di}{\dt} \int\limits_{\Xi(t)} u(t) \, w \dx  + \alpha_{\Pi} \dfrac{\di}{\dt} \int\limits_{\Pi} u(t) \, w \dx \\
	& = \alpha_\Sigma \left[ \, \int\limits_{\Sigma(t)} \pa_t u(t) \, w \dx + \int\limits_{\Sigma(t)} \dive{ \paren{ u(t) \, w \, \vb(t) } } \dx \right] \\
	& \q + \alpha_\Xi \left[ \, \int\limits_{\Xi(t)} \pa_t u(t) \, w \dx + \int\limits_{\Xi(t)} \dive{ \paren{ u(t) \, w \, \vb(t) } } \dx \right] \\
	& \q +  \alpha_{\Pi}  \int\limits_{\Pi}\pa_t u(t) \, w \dx\\
	& =  \int\limits_{\Omega} \alpha(t) \pa_t u(t) \, w \dx + \int\limits_{\Omega} \alpha(t) \dive{  \paren{u(t) \, w \, \vb(t) } } \dx.
\end{align*}
This observation motivates the following definition. 

\begin{definition}
	\label{def:derivative}
	Let $\alpha \ge 0$ be a coefficient that is constant on each subdomain. Given $u \in \Ls^1((0, T), \Hs^1(\Om))$ satisfying $(\alpha(\cdot) u(\cdot), w)_\Om \in \ACs([0, T])$ for all $w
	\in \Hs^1(\Om)$ and that there exists $g \in \Ls^1((0, T), [\Hs^1(\Om)]^\prime)$ such that  for all $ w \in \Hs^1(\Om)$
	\begin{equation}\label{eq:defder_g}
		\dfrac{\di}{\dt} (\alpha(t) u(t), w)_\Om = \inprod{g(t), w}_{1, \Om}, \qqq \text{for a.a. } t\in (0,T).
	\end{equation}
	Then, $\alpha \pa_t u \in \Ls^1((0, T), [\Hs^1(\Om)]^\prime)$ is defined by
	\begin{equation}
		\label{def:alpha_pdt_u}
		\inprod{(\alpha \pa_t u)(t), w}_{1, \Om} := \dfrac{\di}{\dt} (\alpha(t) u(t), w)_\Om - \int\limits_\Om \alpha(t) \dive{\paren{u(t) \, w \, \vb(t)}} \dx, \q \text{for a.a. } t \in (0, T).
	\end{equation}
\end{definition}
Now, we consider the Bochner space
\begin{equation}
	\Ys_\alpha := \brac{u \in \Ls^2((0, T), \Hs^1(\Om)) : \alpha \pa_t u \in \Ls^2((0, T), [\Hs^1(\Om)]^\prime)},
\end{equation}
equipped with its graph norm
\[
\norm{u}^2_{\Ys_\alpha} := \norm{u}^2_{\Ls^2((0, T), \Hs^1(\Om))} + \norm{\alpha \pa_t u}_{\Ls^2((0, T), [\Hs^1(\Om)]^\prime)}^2.
\]
The space $\Ys_\alpha$ is the natural solution space for \eqref{eq:uIBVP}-\eqref{eq:interface_u}. In addition, if $u \in \Ys_\alpha$, it follows from \eqref{def:alpha_pdt_u} that ${\mathcal F}_w\in \Hs^1((0,T))$ for all $w \in \Hs^1(\Om)$. The following lemma provides the density of the space $\Cs^1([0, T], \Hs^1(\Om))$ in $\Ys_\alpha$.

\begin{lemma}
	\label{lem:density}
	The space $\Cs^1([0, T], \Hs^1(\Om))$ is densely contained in $\Ys_{\alpha}$.
\end{lemma}

\begin{proof}
	Let $u \in \Ys_\alpha$. For $0 < \veps \le T/2$, we denote
	\begin{equation}\label{eq:convint}
		u_\veps(t) := \int\limits_0^T \rho_\veps(t + \xi_\veps(t) - s) u(s) \ds, \qqqqqq \xi_\veps(t) := \veps \dfrac{T - 2t}{T}\in \left[-\veps, \veps \right],
	\end{equation}
	where $\rho_\veps : \R \to \R$ is the mollifier defined by
	\[
	\rho_\veps(t) :=
	\begin{cases}
		c\veps^{-1} \exp\paren{t^2/\paren{t^2-\veps^2}} \q &\text{for } \abs{t} < \veps, \\
		0  & \text{elsewhere},
	\end{cases}
	\]
	with $c$ the constant such that $\int_\R \rho_{1}(t) \dt = 1$ (hence, also $\int_\R \rho_{\veps}(t) \dt = 1$). 
	The function $\xi_\veps$ converges to $0$ as $\veps \searrow 0$ and slightly shifts the kernel in \eqref{eq:convint} such that only values of $u$ inside $[0, T]$ are taken into account, see \cite[Figure~16]{Roubicek2005}.
	It is clear that $u_\veps \in \Cs^1([0, T], \Hs^1(\Om))$. We follow \cite[Lemma~7.2]{Roubicek2005} to conclude that $u_\veps$ converges to $u$ in $\Ls^2((0, T), \Hs^1(\Om))$ when $\veps \searrow 0$.
	Next, we shall investigate the convergence of $\alpha \pa_t u_\veps$ in $\Ls^2((0, T), [\Hs^1(\Om)]^\prime)$.
	For all $y \in \Ls^2((0, T), \Hs^1(\Om))$, it holds that
	\begin{align*}
		\int\limits_0^T \inprod{(\alpha \pa_t u_\veps)(t), y(t)}_{1, \Om} \dt 
		& = \int\limits_0^T \int\limits_\Om \alpha(t) \, \pa_t u_\veps(t) \, y(t) \dx \dt \\
		& = \dfrac{T - 2\veps}{T} \int\limits_0^T \int\limits_\Om \int\limits_0^T \alpha(t) \, \rho^\prime_\veps(t + \xi_\veps(t) - s) \, u(s) \, y(t) \ds \dx \dt \\
		& = \dfrac{T - 2\veps}{T} \int\limits_0^T \int\limits_\Om \int\limits_0^T \left(\alpha(t) - \alpha(s)\right) \rho^\prime_\veps(t + \xi_\veps(t) - s) \, u(s) \, y(t) \ds \dx \dt \\
		& \qq + \dfrac{T - 2\veps}{T} \int\limits_0^T \int\limits_\Om \int\limits_0^T \alpha(s) \, \rho^\prime_\veps(t + \xi_\veps(t) - s) \, u(s) \, y(t) \ds \dx \dt \\
		& =: I_1 + I_2.
	\end{align*} 
	For the term $I_1$, we firstly interchange the order of integration between $s$ and $\xb$, then split  integrals over the subdomains $\Sm, \Xi$, and $\Pi$ as follows
	\begin{align*}
		I_1 &= \dfrac{T - 2\veps}{T} \int\limits_0^T \int\limits_0^T  \int\limits_\Om \left(\alpha(t) - \alpha(s)\right) \rho^\prime_\veps(t + \xi_\veps(t) - s) \, u(s) \, y(t)  \dx \ds \dt \\
		& = \dfrac{T - 2\veps}{T} \int\limits_0^T \int\limits_0^T  \int\limits_{\Sigma(t) \cup \Xi(t)\cup \Pi} \alpha(t) \, \rho^\prime_\veps(t + \xi_\veps(t) - s) \, u(s) \, y(t)  \dx \ds \dt \\
		& \qquad - \dfrac{T - 2\veps}{T} \int\limits_0^T \int\limits_0^T  \int\limits_{\Sigma(s) \cup \Xi(s)\cup \Pi}  \alpha(s) \, \rho^\prime_\veps(t + \xi_\veps(t) - s) \, u(s) \, y(t)  \dx \ds \dt \\
		& =  \alpha_\Sm \dfrac{T - 2\veps}{T} \int\limits_0^T \int\limits_0^T \rho_\veps^\prime(t + \xi_\veps(t) - s) \paren{\,\, \int\limits_{\Sm(t)} u(s) \, y(t) \dx - \int\limits_{\Sm(s)} u(s) \, y(t) \dx} \ds \dt \\
		& \qquad + \alpha_\Xi \dfrac{T - 2\veps}{T} \int\limits_0^T \int\limits_0^T \rho_\veps^\prime(t + \xi_\veps(t) - s) \paren{\,\, \int\limits_{\Xi(t)} u(s) \, y(t) \dx - \int\limits_{\Xi(s)} u(s) \, y(t) \dx} \ds \dt \\
		& =: I_{1, \Sm} + I_{1, \Xi}.  
	\end{align*}
	We demonstrate here the calculation for $I_{1, \Sm}$ and note that the term $I_{1, \Xi}$ can be handled similarly. The RTT allows us to deduce that 
	\begin{align*}
		I_{1, \Sm} 
		& \, = \alpha_\Sm \dfrac{T - 2\veps}{T} \int\limits_0^T \int\limits_0^T \rho_\veps^\prime(t + \xi_\veps(t) - s) \int\limits_s^t \dfrac{\di}{\di \eta} \int\limits_{\Sm(\eta)} u(s) \, y(t) \dx \di \eta \ds \dt \\
		& \stackrel{\eqref{eq:Reynolds2}}{=} \alpha_\Sm \dfrac{T - 2\veps}{T} \int\limits_0^T \int\limits_0^T \rho_\veps^\prime(t + \xi_\veps(t) - s) \int\limits_s^t \int\limits_{\Sm(\eta)} \dive{\left(u(s) \, y(t) \, \vb(\eta)\right)} \dx \di \eta \ds \dt \\
		& \, =  \dfrac{T - 2\veps}{T} \int\limits_0^T \int\limits_0^T \rho_\veps^\prime(t + \xi_\veps(t) - s) \int\limits_s^t \int\limits_{\Sm(\eta)} \alpha(\eta) \dive{\left(u(s) \, y(t) \, \vb(\eta)\right)} \dx \di \eta \ds \dt.
	\end{align*}
	Then, summing up $I_{1, \Sm}$ and $I_{1, \Xi}$ gives that
	\[
	I_1 = \dfrac{T - 2\veps}{T} \int\limits_0^T  \int\limits_{t + \xi_\veps(t) - \veps}^{t + \xi_\veps(t) + \veps} \rho_\veps^\prime(t + \xi_\veps(t) - s) \int\limits_s^t \int\limits_\Om \alpha(\eta) \dive{\left(u(s) \, y(t) \, \vb(\eta)\right)} \dx \di \eta \ds \dt =: \dfrac{T - 2\veps}{T} \tilde{I}_1,
	\]
	using $\vb= \zrb $ on $\Pi$. Please note that ${I}_1$ just differs from $\tilde{I}_1$ in the factor $(T - 2\veps)/T,$ which converges to $1$ as $\veps \searrow 0$. Next, using the partial integration w.r.t. $s,$
	we observe that
	\begin{align*}
	&	- \int\limits_0^T \int\limits_\Om \alpha(t) \dive{\left(u(t) \, y(t) \, \vb(t)\right)} \dx \dt \\
	 & = - \int\limits_0^T  \int\limits_{t + \xi_\veps(t) - \veps}^{t + \xi_\veps(t) + \veps}\rho_\veps(t + \xi_\veps(t) - s)  \int\limits_\Om \alpha(t) \dive{\left(u(t) \, y(t) \, \vb(t)\right)} \dx \ds  \dt \\
		& = \int\limits_0^T  \int\limits_{t + \xi_\veps(t) - \veps}^{t + \xi_\veps(t) + \veps} \rho_\veps(t + \xi_\veps(t) - s)  \frac{\di}{\ds }\int\limits_s^t  \int\limits_\Om \alpha(t) \dive{\left(u(t) \, y(t) \, \vb(t)\right)} \dx \di\eta \ds \dt \\
		& = \int\limits_0^T  \int\limits_{t + \xi_\veps(t) - \veps}^{t + \xi_\veps(t) + \veps}\rho^\prime_\veps(t + \xi_\veps(t) - s)  \int\limits_s^t  \int\limits_\Om \alpha(t) \dive{\left(u(t) \, y(t) \, \vb(t)\right)} \dx \di\eta \ds \dt.
	\end{align*}
	Then, subtracting this term from $\tilde{I}_1$ leads us to
	\begin{align*}
		I_3 & := \tilde{I}_1 + \int\limits_0^T  \int\limits_\Om \alpha(t) \dive{\left(u(t) \, y(t) \, \vb(t)\right)} \dx \dt \\
		& \, = \int\limits_0^T  \int\limits_{t + \xi_\veps(t) - \veps}^{t + \xi_\veps(t) + \veps}\rho^\prime_\veps(t + \xi_\veps(t) - s) \int\limits_s^t \int\limits_\Om \left[ \alpha(\eta) \dive{\left(u(s) \, y(t) \, \vb(\eta)\right)} - \alpha(t) \dive{\left(u(t) \, y(t) \, \vb(t)\right)} \right] \dx \di \eta \ds \dt \\
		& \, = \int\limits_0^T \int\limits_{t + \xi_\veps(t) - \veps}^{t + \xi_\veps(t) + \veps} \rho^\prime_\veps(t + \xi_\veps(t) - s) \int\limits_s^t \int\limits_\Om \alpha(\eta) \dive{\left[\paren{u(s) - u(\eta)} y(t) \, \vb(\eta)\right]} \dx \di \eta \ds \dt \\
		& \qq + \int\limits_0^T \int\limits_{t + \xi_\veps(t) - \veps}^{t + \xi_\veps(t) + \veps} \rho^\prime_\veps(t + \xi_\veps(t) - s) \int\limits_s^t \int\limits_\Om \left[\alpha(\eta) \dive{\left(u(\eta) \, y(t) \, \vb(\eta)\right)} - \alpha(t) \dive{\left(u(t) \, y(t) \, \vb(t)\right)} \right] \dx \di \eta \ds \dt \\
		& \, =: I_3^a + I_3^b.
	\end{align*}
	Setting $h := s - t$ and $\lambda := \eta - t$, we can bound $I_3^a$ as follows 
	\begin{align*}
		\abs{I_3^a}^2 
		& \leqc \paren{\int\limits_0^T \norm{y(t)}_{\Hs^1(\Om)} \int\limits_{\xi_\veps(t) - \veps}^{\xi_\veps(t) + \veps} \abs{\rho^\prime_\veps\paren{h - \xi_\veps(t)}} \abs{\, \int\limits_{0}^h \norm{u(t + h) - u(t + \lambda)}_{\Hs^1(\Om)} \di \lambda \, } \di h \dt}^2 \\
		& \leqc \norm{y}^2_{\Ls^2((0, T), \Hs^1(\Om))} \int\limits_0^T \paren{\,\, \int\limits_{\xi_\veps(t) - \veps}^{\xi_\veps(t) + \veps} \abs{\rho^\prime_\veps\paren{h - \xi_\veps(t)}} \abs{\, \int\limits_{0}^h \norm{u(t + h) - u(t + \lambda)}_{\Hs^1(\Om)} \di \lambda \, } \di h}^2 \dt \\
		& \leqc \norm{y}^2_{\Ls^2((0, T), \Hs^1(\Om))} \paren{\int\limits_{- \veps}^{\veps} \rho^\prime_\veps(h)^2 \di h} \int\limits_0^T \,\, \int\limits_{\xi_\veps(t) - \veps}^{\xi_\veps(t) + \veps} \, \paren{\, \int\limits_{0}^h \norm{u(t + h) - u(t + \lambda)}_{\Hs^1(\Om)} \di \lambda}^2 \di h \dt \\
		& \leqc \dfrac{1}{\veps^{3}}  \norm{y}^2_{\Ls^2((0, T), \Hs^1(\Om))}  \int\limits_0^T \,\, \int\limits_{\xi_\veps(t) - \veps}^{\xi_\veps(t) + \veps} \,  \abs{ h }  \int\limits_{\xi_\veps(t) - \veps}^{\xi_\veps(t) + \veps} \norm{u(t + h) - u(t + \lambda)}_{\Hs^1(\Om)}^2 \di \lambda \di h \dt \\
		& \leqc \dfrac{1}{\veps^2} \norm{y}^2_{\Ls^2((0, T), \Hs^1(\Om))} \int\limits_{-2\veps}^{2\veps} \int\limits_{-2\veps}^{2\veps} \int\limits_0^T \norm{u(t + h) - u(t + \lambda)}^2_{\Hs^1(\Om)} \dt \di \lambda \di h \\
		& \leqc \norm{y}^2_{\Ls^2((0, T), \Hs^1(\Om))} \sup_{\abs{h} \le 2\veps} \, \sup_{\abs{\lambda} \le 2\veps} \, \int\limits_0^T \norm{u(t + h) - u(t + \lambda)}^2_{\Hs^1(\Om)} \dt,
	\end{align*}
	where we used the following estimate of the smooth function $\rho^\prime_\veps$:
	\begin{equation}
		\label{eq:estimate_rho_dt}
		\int\limits_{-\veps}^\veps \rho^\prime_\veps(h)^2 \di h = 
		4c^2\veps^2 \int\limits_{-\veps}^\veps \frac{h^2\exp\paren{2 h^2/\paren{h^2-\veps^2}}}{(h^2-\veps^2)^4} \di h
		\stackrel{h=\veps \sin(\theta)}{\leqc} \dfrac{1}{\veps^3}.
	\end{equation}
	Invoking the mean continuity of $u$ as an element in $\Ls^2((0, T), \Hs^1(\Om)$, we are able to pass to the limit for $\veps \searrow 0$ to get that $\lim\limits_{\veps \searrow 0} I_3^a = 0$. In addition to that, the term $I_3^b$ can be rewritten in the following form
	\begin{align*}
		I_3^b & = \int\limits_0^T \int\limits_{t + \xi_\veps(t) - \veps}^{t + \xi_\veps(t) + \veps} \rho^\prime_\veps(t + \xi_\veps(t) - s) \int\limits_s^t \int\limits_\Om \left[\alpha(\eta) \, \nabla u(\eta) \cdot \vb(\eta) -  \alpha(t) \, \nabla u(t) \cdot \vb(t)\right] y(t) \dx \di \eta \ds \dt \\
		& \q + \int\limits_0^T \int\limits_{t + \xi_\veps(t) - \veps}^{t + \xi_\veps(t) + \veps} \rho^\prime_\veps(t + \xi_\veps(t) - s) \int\limits_s^t \int\limits_\Om \left[\alpha(\eta) \, u(\eta) \, \dive{\vb(\eta)} - \alpha(t) \, u(t) \, \dive{\vb(t)}\right] y(t) \dx \di \eta \ds \dt \\
		& \q + \int\limits_0^T \int\limits_{t + \xi_\veps(t) - \veps}^{t + \xi_\veps(t) + \veps} \rho^\prime_\veps(t + \xi_\veps(t) - s) \int\limits_s^t \int\limits_\Om \left[\alpha(\eta) \, u(\eta) \, \vb(\eta) -  \alpha(t) \, u(t) \, \vb(t)\right] \cdot \nabla y(t) \dx \di \eta \ds \dt.
	\end{align*}
	Performing similar estimates as for $I_3^a$, we obtain that
	\begin{multline*}
		\abs{I_3^b}^2 
		\leqc \norm{y}^2_{\Ls^2((0, T), \Hs^1(\Om))}   \sup_{\abs{\lambda} \le 2\veps} \, \left(\int\limits_0^T \norm{(\alpha \nabla u \cdot \vb)(t+\lambda)-(\alpha \nabla u \cdot \vb)(t)}_{\Ls^2(\Om)}^2 \dt \right. \\ \left.  + \int\limits_0^T   \norm{(\alpha u \dive{\vb})(t+\lambda)-(\alpha u \dive{\vb})(t)}_{\Ls^2(\Om)}^2 \dt + \int\limits_0^T \norm{(\alpha u \vb)(t + \lambda) - (\alpha u \vb)(t)}^2_{\LLs^2(\Om)} \dt \right).
	\end{multline*}
	Here, we note that $\alpha \nabla u \cdot \vb$ and $\alpha u \dive{\vb}$ are elements in $\Ls^2((0, T), \Ls^2(\Om)$, while $\alpha u \vb$ is an element of $\Ls^2((0, T), \LLs^2(\Om)$. Thus, using the mean continuity of these elements, we conclude that $\lim\limits_{\veps \searrow 0} I_3^b = 0$, and hence $\lim\limits_{\veps \searrow 0} I_3 = 0$. In other words, the following limit transition holds true: 
	\begin{equation}
		\label{eq:limit_I1}
		\lim\limits_{\veps \searrow 0} I_1 = \lim\limits_{\veps \searrow 0} \tilde{I}_1 = - \int\limits_0^T \int\limits_\Om \alpha(t) \dive{\left(u(t) \, y(t) \, \vb(t)\right)} \dx  \dt.
	\end{equation}
	Next, the convergence of $I_2$ is examined. We first interchange the order of integration between $\xb$ and $s$, then use the partial integration w.r.t. $s$ and the definition \eqref{def:alpha_pdt_u} to see that
	\begin{align*}
		I_2 & = \dfrac{T - 2\veps}{T} \int\limits_0^T  \int\limits_{t + \xi_\veps(t) - \veps}^{t + \xi_\veps(t) + \veps} \, \rho^\prime_\veps(t + \xi_\veps(t) - s) \left( \int\limits_\Om \alpha(s) \, u(s) \, y(t)  \dx \right) \ds \dt \\
		& = \dfrac{T - 2\veps}{T} \int\limits_0^T  \int\limits_{t + \xi_\veps(t) - \veps}^{t + \xi_\veps(t) + \veps} \rho_\veps(t + \xi_\veps(t) - s) \dfrac{\di}{\ds} \paren{\alpha(s) u(s), y(t)}_\Om \ds \dt \\
		& = \dfrac{T - 2\veps}{T} \int\limits_0^T  \int\limits_{t + \xi_\veps(t) - \veps}^{t + \xi_\veps(t) + \veps} \rho_\veps(t + \xi_\veps(t) - s) \inprod{(\alpha \pa_s u)(s), y(t)}_{1, \Om} \ds \dt \\
		& \qq + \dfrac{T - 2\veps}{T} \int\limits_0^T  \int\limits_{t + \xi_\veps(t) - \veps}^{t + \xi_\veps(t) + \veps} \rho_\veps(t + \xi_\veps(t) - s) \int\limits_\Om \alpha(s) \dive{\paren{u(s) \, y(t) \, \vb(s)}} \dx \ds \dt.
	\end{align*}
	Passing into the limit for $\veps \searrow 0$ in $I_2$ with similar reasoning as in $I_1$, then invoking the mean continuity of $\alpha \pa_t u$ as an element in $\Ls^2((0, T), [\Hs^1(\Om)]^\prime)$, we end up with 
	\begin{equation}
		\label{eq:limit_I2}
		\lim_{\veps \searrow 0} I_2  
		= \int\limits_0^T  \inprod{(\alpha \pa_t u)(t), y(t)}_{1, \Om} \dt + \int\limits_0^T \int\limits_\Om \alpha(t) \dive{\left(u(t) \, y(t) \, \vb(t)\right)} \dx \dt. 
	\end{equation}
	Please note that the corresponding arguments for $I_2$ are even simpler than those for $I_1$, since there are only two integrals in time in $I_2$ rather than three in $I_1$. Moreover, instead of the estimate \eqref{eq:estimate_rho_dt} of $\rho^\prime_\veps$, the following estimate of $\rho_{\veps}$ was used
	\[
	\int\limits_{-\veps}^\veps \rho_\veps(h)^2 \di h = 
	\dfrac{c^2}{\veps^2} \int\limits_{-\veps}^\veps \exp\paren{2 h^2/\paren{h^2-\veps^2}}\di h
	\stackrel{h=\veps \sin(\theta)}{\leqc} \dfrac{1}{\veps}.
	\] 
	Now, combining the limit transitions \eqref{eq:limit_I1} and \eqref{eq:limit_I2} leads us to
	\[
	\lim\limits_{\veps \searrow 0} \int\limits_0^T \inprod{(\alpha \pa_t u_\veps)(t), y(t)}_{1, \Om} \dt = \int\limits_0^T \inprod{(\alpha\pa_t u)(t), y(t)}_{1, \Om} \dt,
	\]
	which holds for all $y \in \Ls^2((0, T), \Hs^1(\Om))$. As a consequence, $\alpha \pa_t u_\veps$ converges weakly to $\alpha \pa_t u$ in $\Ls^2((0, T), [\Hs^1(\Om)]^\prime)$ as $\veps \searrow 0$. By the Mazur theorem \cite[Theorem~2 on p.~120]{Yosida1980}, there exists a sequence $\tilde{u}_\veps$ of finite convex combinations of $u_\veps$ such that $\alpha \pa_t \tilde{u}_\veps \to \alpha \pa_t u$ in $\Ls^2((0, T), [\Hs^1(\Om)]^\prime)$. Finally, in conjunction with the strong convergence $u_\veps \to u$ in $\Ls^2((0, T), \Hs^1(\Om))$, we conclude that $\tilde{u}_\veps \to u$ in $\Ys_\alpha$, thereby completing the proof. 
\end{proof}

The next lemma is an extension of the following continuous embedding \cite[Lemma~7.3]{Roubicek2005}:
\[
\brac{u \in \Ls^2((0, T), \Hs^1(\Om)) : \pa_t u \in \Ls^2((0, T), [\Hs^1(\Om)]^\prime)} \emb \Cs([0, T], \Ls^2(\Om)).
\]

\begin{lemma}
	\label{lem:extension}
	If $u \in \Ys_\alpha$, then $\sqrt{\alpha} u \in \Cs([0, T], \Ls^2(\Om))$.
\end{lemma}
\begin{proof}
	Let $u \in \Cs^1([0, T], \Hs^1(\Om))$. For any time interval $(\xi, \eta) \sst (0, T)$, the following identity holds  
	\begin{align*}
		\int\limits_\xi^{\eta} \paren{\alpha(t) \pa_t u(t), u(t)}_{\Sm(t)} \dt  
		& = \dfrac{\alpha_{\Sm}}{2} \int\limits_\xi^{\eta} \int\limits_{\Sm(t)} \pa_t u^2(t) \dx \dt \\
		& \stackrel{\eqref{eq:Reynolds2}}{=} \dfrac{\alpha_{\Sm}}{2} \int\limits_\xi^{\eta} \dfrac{\di}{\dt} \int\limits_{\Sm(t)} u^2(t) \dx \dt - \dfrac{\alpha_{\Sm}}{2} \int\limits_\xi^{\eta} \int\limits_{\Sm(t)} \dive{\paren{\vb(t) \, u^2(t)}} \dx \dt \\
		& =
		\dfrac{\alpha_{\Sm}}{2} \norm{u}^2_{\Ls^2(\Sm)}(\eta) - \dfrac{\alpha_{\Sm}}{2} \norm{u}^2_{\Ls^2(\Sm)}(\xi)
		- \dfrac{\alpha_{\Sm}}{2} \int\limits_\xi^{\eta} \int\limits_{\Sm(t)} \dive{\paren{\vb(t) \, u^2(t)}} \dx \dt.        
	\end{align*}
	Similar identities can be obtained for the integrals over $\Xi(t)$ and $\Pi$. Thus, we arrive at
	\begin{multline*}
	\int\limits_\xi^{\eta} \paren{\alpha(t) \pa_t u(t), u(t)}_{\Omega} \dt
	= \dfrac{1}{2} \norm{\sqrt{\alpha(\eta)}u(\eta)}^2_{\Ls^2(\Omega)}  \\
	- \dfrac{1}{2} \norm{\sqrt{\alpha(\xi)} u(\xi)}^2_{\Ls^2(\Om)} - \dfrac{1}{2} \int\limits_\xi^{\eta} \int\limits_{\Om} \alpha(t) \dive{\paren{\vb(t) \, u^2(t)}} \dx \dt.        
	\end{multline*}
	By a density argument (cf. Lemma~\ref{lem:density}), we can show for $u \in \Ys_\alpha$ that
	\begin{multline}
		\label{eq:continuity}
		\int\limits_\xi^{\eta} \inprod{(\alpha \pa_t u)(t), u(t)}_{1, \Omega} \dt
		= \dfrac{1}{2} \norm{\sqrt{\alpha(\eta)}u(\eta)}^2_{\Ls^2(\Omega)}\\
		 - \dfrac{1}{2} \norm{\sqrt{\alpha(\xi)} u(\xi)}^2_{\Ls^2(\Om)} - \dfrac{1}{2} \int\limits_\xi^{\eta} \int\limits_{\Om} \alpha(t) \dive{\paren{\vb(t) \, u^2(t)}} \dx \dt.        
	\end{multline}
	Finally, we can easily deduce from \eqref{eq:continuity} that
	\[
	\abs{ \ \norm{\sqrt{\alpha(\eta)} u(\eta)}^2_{\Ls^2(\Om)} - \norm{\sqrt{\alpha(\xi)} u(\xi)}^2_{\Ls^2(\Om)}} \leqc \int\limits_{\xi}^{\eta} \norm{(\alpha \pa_t u)(t)}^2_{[\Hs^1(\Om)]^\prime} \dt + \int\limits_{\xi}^{\eta} \norm{u(t)}^2_{\Hs^1(\Om)} \dt.
	\]
	Hence, $\sqrt{\alpha} u \in \Cs([0, T], \Ls^2(\Om))$.
\end{proof}

\section{Uniqueness}
\label{sec:uniqueness}

Now, we are in the position to introduce the variational formulation of the problems \eqref{eq:phiBVP}-\eqref{eq:interface_u}. 
Multiplying the first equations of \eqref{eq:phiBVP}, \eqref{eq:AIBVP} and \eqref{eq:uIBVP} by $\psi \in \Zs, \vphib \in \Wb_0$ and $w \in \Hs^1(\Om)$, respectively, then applying the Green theorem, we arrive at the following variational problem: 
Find $\phi(t) \in \Zs, \Ab(t) \in \Wb_0$ and $u(t) \in \Hs^1(\Om)$ with $(\sm \pa_t \Ab)(t) \in [\HHs^1(\Om)]^\prime$ and $(\alpha \pa_t u)(t) \in [\Hs^1(\Om)]^\prime$ such that 
\begin{equation}
	 \sm_{\Pi} \paren{\nabla\phi(t), \nabla\psi}_{\Pi} + \inprod{j(t), \psi}_{1/2, \Gm} = 0, \label{eq:vf_phi} 
\end{equation}
\begin{multline}
	 \inprod{(\sm \pa_t \Ab)(t), \vphib}_{1, \Om} + \mu_0^{-1} \paren{\curl{\Ab(t)}, \curl{\vphib}}_\Om \\
	  + \sm_{\Pi} \paren{\nabla\phi(t), \vphib}_{\Pi} - \sm_{\Sm} \paren{\vb(t) \times \left(\curl{\Ab(t)}\right), \vphib}_{\Sm(t)} = 0, \label{eq:vf_A} 
\end{multline}
\begin{multline}
	 \inprod{(\alpha \pa_t u)(t), w}_{1, \Om} + \paren{\alpha(t) \vb(t) \cdot \nabla u(t), w}_{\Om} \\ + \paren{\kappa(t) \nabla u(t), \nabla w}_{\Om} = \paren{\RR_r\left(Q(t)\right), w}_{\Theta(t)}, \label{eq:vf_u}
\end{multline}
for all $\psi \in \Zs, \vphib \in \Wb_0$ and $w \in \Hs^1(\Om)$ and for a.a. $t \in (0, T)$.
Note that an equivalent saddle-point formulation of the problem \eqref{eq:vf_A} was introduced in \cite{LSV2021b}, which gives more convenience for the computation. In this paper, however, we use the formulation \eqref{eq:vf_A} for simplicity, and we note that the results obtained in \cite{LSV2021b} are still valid.
In the next step, we summarize all assumptions used in the paper and show the uniqueness of a solution to the variational problem.
\begin{description}
	\item[\text{(AS1)}\namedlabel{as:as1}{\text{(AS1)}}] $\Om$ is an open, bounded, simply-connected domain in $\R^3$ such that $\Om$ is a convex polyhedron or its boundary $\pa\Om$ is of class $\Cs^{1, 1}$. The open connected subdomains $\Sm$ and $\Pi$ are of the class $\Cs^{2,1}$ and separate from each other (see \Cref{sec:model} for more details); 
	\item[\text{(AS2)}\namedlabel{as:as2}{\text{(AS2)}}] The magnetic permeability is a constant on the entire domain $\Om$, and all material coefficients are positive constants on each subdomain, except that the electrical conductivity is vanishing on the air (see \Cref{sec:model} for more details);
	\item[\text{(AS3)}\namedlabel{as:as3}{\text{(AS3)}}] The velocity vector $\vb$ satisfies $\vb \in \CCs^1(\ovl{\Om} \times [0,T]), \vb = \zrb$ in the coil $\Pi$, {and $\vb \cdot \nv = 0$ on the boundary $\pa\Om$};
	\item[\text{(AS4)}\namedlabel{as:as4}{\text{(AS4)}}] $\tilde{u}_0 \in \Hs^1(\Om)$ and $\tilde{\Ab}_0 \in \Wb_0\cap \HHs^2(\Omega)$ satisfying $\curl{\curl{\tilde{\Ab}_0}} = \zrb$ in $\Xi(0)$;
	\item[\text{(AS5)}\namedlabel{as:as5}{\text{(AS5)}}] $j \in \Lips([0,T],\Hs^{-1/2}(\Gm))$.
\end{description}

\begin{remark}
	Each assumption is only needed for some specific analyses. However, all assumptions \ref{as:as1}-\ref{as:as5} together are necessary and sufficient for all analyses throughout the paper to be valid. Therefore, we list them in the same place for ease of readability. In the following, some comments on those assumptions are provided:
	\begin{itemize}
		\item The analysis can be extended for the case when the magnetic permeability $\mu_0$ and the thermal conductivity $\kappa_\Sm, \kappa_\Pi, \kappa_\Theta$ are Lipschitz continuous in both space and time, while the material coefficients $\sm$ and $\alpha$ on each subdomain are of class $\Cs^{1, 1}$ in both space and time. 
		\item If one would consider strongly coupled induction heating systems (e.g., the electrical conductivity $\sm$ depends on the temperature $u$), then the uniqueness of a solution might be sacrificed, as noted in \cite{CGS2017,LSV2022b}.
		\item The condition $\vb \cdot \nv = 0$ on $\pa\Om$ is not needed for our analysis to be valid. However, it makes sense since we model the entire domain $\Om$ as an isolated system, i.e., there is no exchange of mass or energy with the exterior.
		\item {In the problem \eqref{eq:AIBVP}, the initial datum $\Ab(0)$ is only given on the conductors $\Theta(0)$ because the electrical conductivity $\sm$ vanishes on the air. However, further results in this paper require $\Ab(0)$ to be extended to the entire domain $\Om$. To do so, we invoke the result in \cite[Proposition~4.1]{KMPT2000} to show that if $\tilde{\Ab}_0$ satisfies
			\[
			\tilde{\Ab}_0 \in \HHs^2(\Theta(0)), \qqq \dive{\tilde{\Ab}_0} = 0 \q \text{in} \q \Theta(0), \qqq \tilde{\Ab}_0 = \zrb \q \text{on} \q \Gm,
			\]
			then there exists an extension $\tilde{\Ab}_0 \in \HHs^2(\Om) \cap \HHs^1_0(\Om)$ with $\dive{\tilde{\Ab}_0} = 0$.}
		The strong regularity of the initial guest $\tilde{\Ab}_0$ in \ref{as:as4} is only needed for the interior regularity of the discrete solution $\Ab_i$ (see \Cref{lem:higher_regularity} or \cite{LSV2021b} for more details). However, this regularity is crucial in proving the convergence of the temporal discretization of the complete problem (see \Cref{thm:existence_phi_A,thm:existence_u}).
	\end{itemize}
\end{remark}
Next, we show the uniqueness of a solution to the variational problem \eqref{eq:vf_phi}-\eqref{eq:vf_u}.

\begin{theorem}[Uniqueness]\label{thm:uniqueness}
	Let the assumptions \ref{as:as1}-\ref{as:as5} be satisfied. 
	Then, the variational system \eqref{eq:vf_phi}-\eqref{eq:vf_u} admits at most one solution $\paren{\phi, \Ab, u}$ satisfying $\phi \in \Ls^2((0, T), \Zs), \Ab \in \Ls^{2}((0, T), \Wb_0)$ with $\pa_t\Ab(t) \in \LLs^2(\Theta(t))$ for a.a. $t\in (0,T)$, and $u \in \Ys_\alpha$. 
\end{theorem}

\begin{proof}
	We assume that there exist two solutions $(\phi_1, \Ab_1, u_1)$ and $(\phi_2, \Ab_2, u_2)$ to the variational equations \eqref{eq:vf_phi}-\eqref{eq:vf_u}. Then, the solution $(\phi, \Ab)$, with $\phi = \phi_1 - \phi_2$ and $\Ab = \Ab_1 - \Ab_2$, solves the linear system \eqref{eq:vf_phi}-\eqref{eq:vf_A} with given data $j = 0$ and $\tilde{\Ab}_0 = \zrb$. By means of \cite[Theorem~3.1]{LSV2021a}, we get that $\phi = 0$ and $\Ab = \zrb$ in the corresponding spaces. This result implies that $u = u_1 - u_2$ also fulfills \eqref{eq:vf_u} with $\tilde{u}_0 = 0$ and $Q = 0$. Setting $w = u(t)$ in \eqref{eq:vf_u} and then integrating in time over $(0, \xi) \sst (0, T)$ gives us that
	\begin{equation}
		\label{eq:c}
		\int\limits_0^{\xi} \inprod{(\alpha \pa_t u)(t), u(t)}_{1, \Om} \dt + \int\limits_0^{\xi} \paren{\alpha(t) \vb(t) \cdot \nabla u(t), u(t)}_{\Om} \dt + \int\limits_0^{\xi} \paren{\kappa(t) \nabla u(t), \nabla u(t)}_{\Om} \dt = 0.
	\end{equation}
	We can immediately see that 
	\[
	\left| \int\limits_0^{\xi} \paren{\alpha(t) \vb(t) \cdot \nabla u(t), u(t)}_{\Om} \dt \,\right| 
	\leq  \veps \int\limits_0^{\xi} \norm{\nabla u(t)}^2_{\LLs^2(\Om)} \dt + C_{\veps} \int\limits_0^{\xi} \norm{\sqrt{\alpha(t)} u(t)}^2_{\Ls^2(\Om)} \dt,
	\]
	and
	\[
	\int\limits_0^{\xi} \paren{\kappa(t) \nabla u(t), \nabla u(t)}_{\Om} \dt \geq \min\{\kappa_{\Sm}, \kappa_{\Pi}, \kappa_{\Xi} \} \int\limits_0^{\xi} \norm{\nabla u(t)}^2_{\LLs^2(\Om)} \dt. 
	\]
	The first integral on the left-hand side (LHS) of \eqref{eq:c} can be rewritten as follows
	\begin{multline*}
	\int\limits_0^{\xi} \inprod{(\alpha \pa_t u)(t), u(t)}_{1, \Omega} \dt
	\stackrel{\eqref{eq:continuity}}{=} \dfrac{1}{2} \norm{\sqrt{\alpha(\xi)}u(\xi)}^2_{\Ls^2(\Omega)}\\
	 - \int\limits_0^{\xi} \paren{\alpha(t)\nabla u(t) \cdot \vb(t), u(t)}_{\Omega} \dt 
	- \dfrac{1}{2} \int\limits_0^{\xi} \paren{\alpha(t) u(t) \dive{\vb(t)}, u(t)}_{\Omega} \dt.        
	\end{multline*}
	The integrals on the right-hand side (RHS) of this identity can be handled as above.
	Therefore, we arrive at
	\[
	\norm{\sqrt{\alpha(\xi)} u(\xi)}^2_{\Ls^2(\Om)} + (1-\veps)  \int\limits_0^{\xi} \norm{\nabla u(t)}^2_{\LLs^2(\Om)} \dt \leqslant C_{\veps} \int\limits_0^{\xi} \norm{\sqrt{\alpha(t)}u(t)}^2_{\Ls^2(\Om)} \dt. 
	\]
	Finally, fixing a sufficiently small $\veps > 0$ and then applying a Gr\"{o}nwall argument shows that $u = 0$ in $\Ls^2((0, T), \Hs^1(\Om))$. Note that the Gr\"{o}nwall argument can be done thanks to the continuity in time of $\norm{\sqrt{\alpha} u}_{\Ls^2(\Om)}$ (cf. \Cref{lem:extension}).
\end{proof}

\section{Time discretization}
\label{sec:TD}

In this section, we design a time-discrete approximation scheme based on the backward Euler method for solving the variational system. The time interval $[0, T]$ is equidistantly partitioned into $n \in \Z^+$ subintervals with time step $\tau = \frac{T}{n}$. At time-point $t_i = i\tau, i = 1, 2, \ldots , n$, we introduce the following notations for any function $f$ and any time-dependent domain $\om$
\[
f_i = f(t_i), \qqqq \delta f_i = \dfrac{f_i - f_{i-1}}{\tau}, \qqqq \om_i = \om(t_i).
\]
Starting from the initial data $\tilde{\Ab}_0$ and $\tilde{u}_0$, we find the solution $\phi_i\in \Zs, \Ab_i \in \Wb_0$ and $u_i \in \Hs^1(\Om)$, with $i = 1, 2, \ldots, n$, such that the following identities are valid for all $\psi \in \Zs, \vphib \in \Wb_0$ and $w \in \Hs^1(\Om)$
\begin{equation}
	 \sm_{\Pi} \paren{\nabla\phi_i, \nabla\psi}_{\Pi} + \inprod{j_i, \psi}_{1/2, \Gm} = 0, \label{eq:TD_phi} 
	 \end{equation}
	\begin{multline}
	 \paren{\sm_i \delta \Ab_i, \vphib}_{\Theta_i} + \mu^{-1}_0 \paren{\curl{\Ab_i}, \curl{\vphib}}_\Om\\ + \sm_{\Pi} \paren{\nabla\phi_i, \vphib}_{\Pi} - \sm_{\Sm} \paren{\vb_i \times (\curl{\Ab_i}), \vphib}_{\Sm_i} = 0, \label{eq:TD_A} 
	\end{multline}
	\begin{equation} \paren{\alpha_i \delta u_i, w}_{\Om} + \paren{\alpha_i \vb_i \cdot \nabla u_i, w}_{\Om} + \paren{\kappa_i \nabla u_i, \nabla w}_{\Omega} = \paren{\RR_r(Q_i), w}_{\Theta_i}, \label{eq:TD_u}
\end{equation}
where 
\[
Q_i = \sigma_i \abs{ \delta\Ab_i + \chi_{\Pi} \nabla\phi_i - \vb_i \times (\curl{\Ab}_i)}^2. 
\]
At each iteration step $i$, the equation \eqref{eq:TD_phi} is firstly solved, then followed by \eqref{eq:TD_A} and \eqref{eq:TD_u}, respectively. In the next lemma, the solvability of the time discretization system will be proven.

\begin{lemma}[Solvability] 
	\label{lem:solvability}
	Let the assumptions \ref{as:as1}-\ref{as:as5} be fulfilled. Then, $\phi_0 \in \Zs$ exists uniquely. Moreover, there exists a positive constant $\tau_0$ such that for any $i = 1, 2, \ldots, n$, with $\tau = \frac{T}{n} < \tau_0$, there exists a unique triplet  $\left(\phi_i, \Ab_i, u_i \right) \in \Zs \times \Wb_0 \times \Hs^1(\Om)$ solving the system \eqref{eq:TD_phi}-\eqref{eq:TD_u}.
\end{lemma}

\begin{proof}
	The proof of the solvability of the system \eqref{eq:TD_phi}-\eqref{eq:TD_A} can be adopted from \cite[Lemma~4.1]{LSV2021b}, so we omit this part.
	Let us define a bilinear form $e_i : \Hs^1(\Om) \times \Hs^1(\Om) \to \R$, with $i = 1, 2, \ldots, n$, such that
	\[
	e_i(u, w) = \frac{1}{\tau} \paren{\alpha_i u, w}_{\Om} + \paren{\alpha_i \vb_i \cdot \nabla u, w}_{\Om} + \paren{\kappa_i \nabla u, \nabla w}_{\Omega}.
	\]
	Then, the variational problem \eqref{eq:TD_u} can be rewritten as follows
	\begin{equation}
		\label{eq:TD_u_1}
		e_i(u_i, w) = \frac{1}{\tau} \paren{\alpha_i u_{i-1}, w}_{\Om} + \paren{\RR_r(Q_i), w}_{\Theta_i}.
	\end{equation}
	We can easily get that
	\[
	e_i(u, w) \leqc \norm{u}_{\Hs^1(\Om)} \norm{w}_{\Hs^1(\Om)},
	\]
	which implies the boundedness of the form $e_i$. For any $i = 1, 2, \ldots, n$, the Cauchy-Schwarz and $\veps$-Young inequalities allow us to show that 
	\begin{align*}
		e_i(u, u) & = \frac{1}{\tau} \paren{\alpha_i u , u}_{\Om} + \paren{\alpha_i \vb_i \cdot \nabla u, u}_{\Om} + \paren{\kappa_i \nabla u, \nabla u}_{\Omega} \\
		& \geqc \paren{\frac{1}{\tau} - C_{\veps}} \norm{u}^2_{\Ls^2(\Om)} + (1 - \veps) \norm{\nabla u}^2_{\LLs^2(\Om)}.
	\end{align*}
	We fix a sufficiently small $\veps > 0$, then choose a sufficiently small time step $\tau < \tau_0$ to claim that the form $e_i$ is $\Hs^1(\Om)$-elliptic.
	Since $u_{i-1} \in \Hs^1(\Om)$ is given and $\RR_r(Q_i)$ is bounded by the constant $r$, the RHS of \eqref{eq:TD_u_1} defines a bounded linear functional on $\Hs^1(\Om)$. As a consequence, there exists a unique solution $u_i \in \Hs^1(\Om)$ to the problem \eqref{eq:TD_u} for any $i = 1, 2, \ldots, n$, according to the Lax-Milgram lemma \cite[Theorem~18.E]{Zeidler-IIA1990}.
\end{proof}

Now, some a priori estimates for iterates will be investigated. The following stability estimate for the solution $\Ab_i$ is directly derived from \cite[Lemma~4.3]{LSV2021b}.
\begin{lemma}[A priori estimate for $\Ab_i$]
	\label{lem:est_A}
	Let the assumptions \ref{as:as1}-\ref{as:as5} be fulfilled. Then, there exist positive constants $\tau_0$ and $C$ such that for any $\tau < \tau_0$, there holds that
	\begin{multline}
		\label{eq:est_A}
		\max\limits_{1 \leq l \leq n} \norm{\delta\Ab_l}^2_{\LLs^2(\Theta_l)} + \max\limits_{1 \le l \le n} \norm{\curl \Ab_l}^2_{\LLs^2(\Om)} \\ + \sum\limits_{i=1}^n \norm{\curl \delta\Ab_i}^2_{\LLs^2(\Om)} \tau + \sum\limits_{i=1}^n \norm{\delta\Ab_i - \delta\Ab_{i-1}}^2_{\LLs^2(\Theta_{i-1})} \le C. 
	\end{multline}
\end{lemma}
This a priori estimate was thoroughly proved in \cite{LSV2021b}. It is noteworthy that the proof relies on the following property of the solution $\Ab_i$.
\begin{lemma}[Higher interior regularity]
	\label{lem:higher_regularity}
	Let the assumptions \ref{as:as1}-\ref{as:as5} be fulfilled. Then, for any $i = 1, 2, \ldots, n$, $\curl{\Ab_i} \in \HHs^1(\Om^\prime)$ for any subset $\Om^\prime \sst \sst \Om$ (i.e., $\ovl{\Om^\prime} \sst \Om$). Moreover, there exists a constant $C(\Om^{\prime}) > 0$ such that
	\[
	\norm{\curl{\Ab_i}}_{\HHs^1(\Om^\prime)} \le C \paren{\norm{\delta \Ab_i}_{\LLs^2(\Theta_i)} + \norm{\curl{\Ab_i}}_{\LLs^2(\Om)} + \norm{\nabla \phi_i}_{\LLs^2(\Pi)}}.
	\]
\end{lemma}
The proof of this interior regularity was provided in \cite{LSV2021b}, which unfortunately had a mistake (it is not justified to consider $\curl{\curl{\Ab_i}}$ as an element in $\LLs^2(\Om)$ since $\CCs^\infty_0(\Om) \not\subset \Wb_0$). In the following, we give a corrected proof of \Cref{lem:higher_regularity} that is adopted from \cite[Lemma~4.1.2]{Le2022}.
\begin{proof}[Proof of \Cref{lem:higher_regularity}]
	For any $i = 1, 2, \ldots, n$, let us denote 
	\[
	\pb_i := \mu_0 \paren{\sm_{\Sm} \vb_i \times (\curl{\Ab_i}) - \sm_i \delta \Ab_i - \chi_{\Pi} \sm_{\Pi} \nabla \phi_i} \in \LLs^2(\Om).
	\]
	Because $\Ab_i \in \Wb_0$, the functional $\curl{\curl{\Ab_i}} \in \HHs^{-1}(\Om)$, where $\HHs^{-1}(\Om)$ is the dual space of $\HHs^1_0(\Om)$. Then, the equation \eqref{eq:TD_A} implies that
	\[
	\inprod{\curl{\curl{\Ab_i}} - \pb_i, \vphib} = 0 \qqqq \fa \vphib \in \Wb_0 \cap \HHs^1_0(\Om),
	\]
	where the duality pairing is between $\HHs^{-1}(\Om)$ and $\HHs^1_0(\Om)$.
	Hence, according to \cite[Lemma~2.1 on p.~22]{Girault1986}, there exists a scalar function $v_i \in \Ls^2(\Om)$ such that 
	\[
	\curl{\curl{\Ab_i}} = \pb_i + \nabla v_i.
	\]
	Now, let $\Bb_i := \curl{\Ab_i}$. The field $\Bb_i \in \LLs^2(\Om)$ satisfies $\dive{\Bb_i} = 0$ and
	\begin{align*}
		- \Delta \Bb_i & = \curl{\curl{\Bb_i}} - \nabla(\dive{\Bb_i}) \\
		& = \curl{\curl{\curl{\Ab_i}}} \\
		& = \curl{\pb_i} + \curl{\nabla v_i} \\
		& = \curl{\pb_i}. 
	\end{align*}
	Since $\curl{\pb_i} \in \HHs^{-1}(\Om)$, we follow \cite[Lemma~3]{DF2020} to get that $\Bb_i \in \HHs^1(\Om^\prime)$ or $\curl{\Ab_i} \in \HHs^1(\Om^\prime)$ for any subset $\Om^\prime \sst \sst \Om$. Next, we adopt the technique in \cite[Theorem~1 on p.~309]{Evans1997} to acquire the estimate of $\curl{\Ab_i}$ in $\HHs^1(\Om^\prime)$. We firstly fix a subdomain $\Om^\prime \sst \sst \Om$, and then choose $\Om^{\star}$ such that $\Om^\prime \sst \sst \Om^{\star} \sst \sst \Om$. Restricting the test function $\vphib \in \brac{\fb \in \CCs^{\infty}_0(\Om^{\star}) : \dive{\fb} = 0} \sst \Wb_0$ in the equation \eqref{eq:TD_A} leads us to that
	\begin{equation}
		\label{eq:useful_5}
		\paren{\curl{\curl{\Ab_i}}, \vphib}_{\Om^\star} = \paren{\pb_i, \vphib}_{\Om^\star}.
	\end{equation}
	In virtue of the density argument in \cite[Theorem~2.8 on p.~30]{Girault1986}, the relation \eqref{eq:useful_5} is still valid for any $\vphib \in \HHs_0(\textup{div}_0, \Om^\star)$, where
	\[
	\HHs_0(\textup{div}_0, \Om^\star) = \brac{\fb \in \LLs^2(\Om^{\star}) :  \dive{\fb} = 0, \ \left.\fb\right|_{\pa\Om^{\star}} \cdot \nv = 0}.
	\]
	Now, let $\gamma \in \Cs^\infty_0(\Om^{\star})$ such that $\gamma = 1$ in $\Om^{\prime}$. 
	Since $\gamma^2 \curl{\Ab_i} \in \HHs^1_0(\Om^{\star})$, we invoke \cite[Remark~2.5 on p.~35]{Girault1986} to get that $\curl{(\gamma^2 \curl{\Ab_i})} \in \HHs_0(\textup{div}_0, \Om^\star)$. Hence, setting $\vphib = \curl{(\gamma^2 \curl{\Ab_i})}$ in \eqref{eq:useful_5} implies that
	\[
	\paren{\curl{\curl{\Ab_i}}, \curl{(\gamma^2 \curl{\Ab_i})}}_{\Om^\star} = \paren{\pb_i, \curl{(\gamma^2 \curl{\Ab_i}})}_{\Om^\star}.
	\]
	Using the Cauchy-Schwarz and $\veps$-Young inequalities together with the following identity
	\[
	\curl{\left(\gamma^2 \curl{\Ab_i}\right)} = \gamma^2 \curl{\curl{\Ab_i}} + 2\gamma \nabla \gamma \times (\curl{\Ab_i}),
	\]
	we arrive at
	\[
	\norm{\gamma \curl{\curl{\Ab_i}}}_{\LLs^2(\Om^\star)} \leqc \norm{\pb_i}_{\LLs^2(\Om)} + \norm{\curl{\Ab_i}}_{\LLs^2(\Om)}.
	\]
	Finally, we use the fact
	\[
	\norm{\nabla \fb}^2_{\LLs^2(\Om)} = \norm{\curl{\fb}}^2_{\LLs^2(\Om)} + \norm{\dive{\fb}}^2_{\Ls^2(\Om)} \qqqq \fa \fb \in \HHs^1_0(\Om),
	\]
	to deduce that
	\begin{align*}
		\norm{\curl{\Ab_i}}_{\HHs^1(\Om^\prime)} 
		& \le \norm{\gamma \curl{\Ab_i}}_{\HHs^1(\Om^\star)} \\
		& \leqc \norm{\curl{(\gamma \curl{\Ab_i})}}_{\LLs^2(\Om^\star)} + \norm{\dive{(\gamma \curl{\Ab_i})}}_{\Ls^2(\Om^\star)} \\
		& \le \norm{\gamma \curl{\curl{\Ab_i}}}_{\LLs^2(\Om^\star)} + \norm{\nabla \gamma \times (\curl{\Ab_i})}_{\LLs^2(\Om^\star)} + \norm{\nabla\gamma \cdot (\curl{\Ab_i})}_{\Ls^2(\Om^\star)} \\
		& \leqc \norm{\pb_i}_{\LLs^2(\Om)} + \norm{\curl{\Ab_i}}_{\LLs^2(\Om)} \\
		& \leqc \norm{\delta \Ab_i}_{\LLs^2(\Theta_i)} + \norm{\curl{\Ab_i}}_{\LLs^2(\Om)} + \norm{\nabla \phi_i}_{\LLs^2(\Pi)},
	\end{align*}
	which allows us to accomplish the proof.
\end{proof}

The next lemma provides a priori estimate for the discrete solutions $u_i$.
\begin{lemma}[A priori estimate for $u_i$]
	\label{lem:est_u}
	Let the assumptions \ref{as:as1}-\ref{as:as5} be fulfilled. Then, there exist positive constants $C$ and $\tau_0$ such that for any $\tau < \tau_0$, the following relation holds true
	\begin{equation}
		\label{eq:est_u}
		\max\limits_{1 \leq l \leq n} \norm{u_l}^2_{\Ls^2(\Om)} + \sum\limits_{i=1}^n \norm{\nabla u_i}^2_{\LLs^2(\Om)}\tau + \sum\limits_{i=1}^n \norm{u_i - u_{i-1}}^2_{\Ls^2(\Om)} {+ \sum\limits_{i=1}^n \norm{\alpha_i \delta u_i}^2_{[\Hs^1(\Om)]^\prime} \tau} \le C.
	\end{equation}
\end{lemma}

\begin{proof}
	We set $w = u_i \tau$ in the equation \eqref{eq:TD_u} and sum the result up to $1 \le l \le n$ to get that
	\begin{multline}
		\label{eq:est_u_testing}
		\sum\limits_{i=1}^l \paren{\alpha_i (u_i - u_{i-1}), u_i}_{\Om} + \sum\limits_{i=1}^l \paren{\alpha_i \vb_i \cdot \nabla u_i, u_i}_{\Om} \tau\\
		 + \sum\limits_{i=1}^l \paren{\kappa_i \nabla u_i, \nabla u_i}_{\Omega} \tau = \sum\limits_{i=1}^l \paren{\RR_r(Q_i), u_i}_{\Theta_i} \tau.
	\end{multline}
	Firstly, we rearrange the first term on the LHS as follows
	\begin{multline}
		\label{eq:Abel}
			2 \sum\limits_{i=1}^l \paren{\alpha_i (u_i - u_{i-1}), u_i}_{\Om}  = \sum\limits_{i=1}^l \paren{\alpha_i u_i, u_i}_\Om - \sum\limits_{i=1}^l \paren{\alpha_{i-1} u_{i-1}, u_{i-1}}_\Om \\
			+ \sum\limits_{i=1}^l \paren{\alpha_i (u_i - u_{i-1}), u_i - u_{i-1}}_\Om - \sum\limits_{i=1}^l \paren{\paren{\alpha_i - \alpha_{i-1}} u_{i-1}, u_{i-1}}_\Om.
	\end{multline}
	The last term on the RHS of \eqref{eq:Abel} can be split over the subdomains in the following way
	\begin{align*}
		\sum\limits_{i=1}^l \paren{\paren{\alpha_i - \alpha_{i-1}} u_{i-1}, u_{i-1}}_\Om
		& = \sum\limits_{i=1}^l \paren{\alpha_i u_{i-1}, u_{i-1}}_\Om - \sum\limits_{i=1}^l \paren{\alpha_{i-1} u_{i-1}, u_{i-1}}_\Om \\
		& = \sum\limits_{i=1}^l \paren{\alpha_i u_{i-1}, u_{i-1}}_{\Sigma_i \cup \Xi_i \cup \Pi} - \sum\limits_{i=1}^l \paren{\alpha_{i-1} u_{i-1}, u_{i-1}}_{\Sigma_{i-1} \cup \Xi_{i-1} \cup \Pi}.
	\end{align*}
	Then, the RTT can be used to estimate the integrals over the workpiece as
	\begin{align*}
		\abs{\paren{\alpha_i u_{i-1}, u_{i-1}}_{\Sm_i} - \paren{\alpha_{i-1} u_{i-1}, u_{i-1}}_{\Sm_{i-1}}}
		& = \alpha_{\Sm} \abs{\, \int\limits_{t_{i-1}}^{t_i} \dfrac{\mbox{d}}{\dt} \int\limits_{\Sm(t)} u_{i-1}^2 \dx \dt} \\
		& \stackrel{\eqref{eq:Reynolds1}}{=} \alpha_{\Sm} \abs{\,\int\limits_{t_{i-1}}^{t_i} \int\limits_{\pa\Sm(t)} u_{i-1}^2 (\vb \cdot \nv)(t) \ds \dt } \\
		& \stackrel{\eqref{eq:Necas_like}}{\le} \veps \norm{\nabla u_{i-1}}^2_{\LLs^2(\Om)} \tau + C_{\veps} \norm{u_{i-1}}^2_{\Ls^2(\Om)} \tau.
	\end{align*}
	A similar estimate can be deduced for the integrals over the air subdomains $\Xi_i$ and $\Xi_{i-1}$, while the corresponding terms vanish on the coil $\Pi$. Hence, we are able to obtain from \eqref{eq:Abel} that (see also \cite[Lemma~2.3]{Slodicka2021})
	\begin{multline*}
	\sum\limits_{i=1}^l \paren{\alpha_i (u_i - u_{i-1}), u_i}_{\Om} \geqc \norm{u_l}^2_{\Ls^2(\Om)} - C \norm{\tilde{u}_0}^2_{\Hs^1(\Om)} \\
	+ \sum\limits_{i=1}^l \norm{u_i - u_{i-1}}^2_{\Ls^2(\Om)} - \veps \sum\limits_{i=1}^{l-1} \norm{\nabla u_{i}}^2_{\LLs^2(\Om)} \tau -  C_{\veps} \sum\limits_{i=1}^{l-1} \norm{u_{i}}^2_{\Ls^2(\Om)} \tau.
	\end{multline*}
	Next, the third term on the LHS of \eqref{eq:est_u_testing} can be bounded by
	\[
	\sum\limits_{i=1}^l \paren{\kappa_i \nabla u_i, \nabla u_i}_{\Om} \tau \geq \min\{\kappa_{\Sm}, \kappa_{\Pi}, \kappa_{\Xi}\}\sum\limits_{i=1}^l \norm{\nabla u_i}^2_{\LLs^2(\Om)} \tau.
	\]
	The Cauchy-Schwarz and $\veps$-Young inequalities can be used to handle the remaining terms of \eqref{eq:est_u_testing} as follows
	\begin{align*}
		& \abs{ \sum\limits_{i=1}^l \paren{\alpha_i \vb_i \cdot \nabla u_i, u_i}_{\Om} } \tau \le \veps \sum\limits_{i=1}^l \norm{\nabla u_i}^2_{\LLs^2(\Om)} \tau + C_{\veps} \sum\limits_{i=1}^l \norm{u_i}^2_{\Ls^2(\Om)} \tau, \\
		& \abs{ \sum\limits_{i=1}^l \paren{\RR_r(Q_i), u_i}_{\Theta_i} } \tau \leqc \sum\limits_{i=1}^l \norm{u_i}^2_{\Ls^2(\Om)} \tau + r^2 \sum\limits_{i=1}^l \tau
		\leqc \sum\limits_{i=1}^l \norm{u_i}^2_{\Ls^2(\Om)}\tau + 1.
	\end{align*}
	Collecting all estimates above, we arrive at 
	\[
	\norm{u_l}^2_{\Ls^2(\Om)}  + \sum\limits_{i=1}^l \norm{u_i - u_{i-1}}^2_{\Ls^2(\Om)}  + (1- \veps) \sum\limits_{i=1}^{l} \norm{\nabla u_{i}}^2_{\LLs^2(\Om)} \tau 
	\leqc 1 + C_{\veps} \sum\limits_{i=1}^{l} \norm{u_{i}}^2_{\Ls^2(\Om)} \tau.
	\]
	{Fixing a sufficiently small $\veps > 0$ and applying the Gr\"{o}nwall argument, then taking the maximum of two sides over $1 \le l \le n$ leads us to that
		\begin{equation}
			\label{eq:est_u_1}
			\max\limits_{1 \leq l \leq n} \norm{u_l}^2_{\Ls^2(\Om)} + \sum\limits_{i=1}^n \norm{\nabla u_i}^2_{\LLs^2(\Om)}\tau + \sum\limits_{i=1}^n \norm{u_i - u_{i-1}}^2_{\Ls^2(\Om)} \leqc 1. 
		\end{equation}
		Finally, using the inequality \eqref{eq:est_u_1} and the definition
		\[
		\norm{\alpha_i \delta u_i}_{[\Hs^1(\Om)]^\prime} = \sup_{w \in \Hs^1(\Om), \, w \neq 0} \dfrac{\paren{\alpha_i \delta u_i, w}_\Om}{\norm{w}_{\Hs^1(\Om)}},
		\]
		we can easily show that
		\[
		\sum\limits_{i=1}^n \norm{\alpha_i \delta u_i}^2_{[\Hs^1(\Om)]^\prime} \tau \leqc 1,
		\]
		which concludes the proof.}
\end{proof}

\section{Existence of a solution}
\label{sec:existence}

This section is the main part of the paper, concerning the existence of a solution to the variational system as well as the convergence of the proposed numerical scheme. Firstly, we introduce some piecewise-constant- and piecewise-affine-in-time functions and subdomains
\begin{align*}
	& \ovl{j}_n(t) = j_i, && \ovl{\vb}_n(t) = \vb_i, && \\
	& \ovl{\sm}_n(t) = \sm_i, && \ovl{\kappa}_n(t) = \kappa_i, && \ovl{\alpha}_n(t) = \alpha_i, \\
	& \widetilde{\Sm}_n(t) = \Sm_i, && \widetilde{\Theta}_n(t) = \Theta_i, && \widetilde{\Xi}_n(t) = \Xi_i, \\
	& \ovl{\phi}_n(t) = \phi_i, && \ovl{\Ab}_n(t) = \Ab_i, && \Ab_n(t) = \Ab_{i-1} + \paren{t - t_{i-1}} \delta \Ab_i, \\
	& \ovl{u}_n(t) = u_i, && \unl{u}_n(t) = u_{i-1}, &&  u_n(t) = u_{i-1} + (t - t_{i-1}) \delta u_i,
\end{align*}
for all $t \in (t_{i-1}, t_i]$, with $i = 1, 2, \ldots, n$. The value of the continuous functions $\Ab_n$ and $u_n$ at time $t = 0$ are given by
\[
\Ab_n(0) = \tilde{\Ab}_0, \qqqqqq u_n(0) = \tilde{u}_0.
\]
In addition, the following piecewise-constant Joule heat source is defined for all $t \in (0, T]$ 
\[
\ovl{Q}_n(t) = \ovl{\sm}_n(t) \abs{\pa_t\Ab_n(t) + \chi_{\Pi} \nabla \ovl{\phi}_n(t) - \ovl{\vb}_n(t) \times \left(\curl{\ovl{\Ab}_n}(t)\right)}^2.
\]
Now, we can rewrite the time-discrete equations \eqref{eq:TD_phi}-\eqref{eq:TD_u} as follows
\begin{equation}
 \sm_{\Pi} \paren{\nabla\ovl{\phi}_n(t), \nabla\psi}_{\Pi} + \inprod{\ovl{j}_n(t), \psi}_{1/2, \Gm} = 0, \label{eq:Rothe_phi} 
\end{equation}
\begin{multline}
 \paren{\ovl{\sm}_n(t)\pa_t \Ab_n(t), \vphib}_{\widetilde{\Theta}_n(t)} + \mu_0^{-1} \paren{\curl{\ovl{\Ab}_n(t)}, \curl{\vphib}}_\Om \\
  + \sm_{\Pi} \paren{\nabla\ovl{\phi}_n(t), \vphib}_{\Pi}- \sm_{\Sm} \paren{\ovl{\vb}_n(t)\times \left(\curl{\ovl{\Ab}_n}(t)\right), \vphib}_{\widetilde{\Sm}_n(t)} = 0, \label{eq:Rothe_A} 
\end{multline}
\begin{multline}
	 \paren{\ovl{\alpha}_n(t) \pa_t u_n(t), w}_{\Om} + \paren{\ovl{\alpha}_n(t) \ovl{\vb}_n(t) \cdot \nabla \ovl{u}_n(t), w}_{\Om} \\
	 + \paren{\ovl{\kappa}_n(t) \nabla \ovl{u}_n(t), \nabla w}_{\Omega} = \paren{\RR_r\left(\ovl{Q}_n(t)\right), w}_{\widetilde{\Theta}_n(t)}, \label{eq:Rothe_u} 
\end{multline}
which are valid for all $\psi \in \Zs, \vphib \in \Wb_0$ and $w \in \Hs^1(\Om)$, and for all $t \in (0, T]$. The following lemma shows the convergence of the piecewise-constant approximation of the given data.
\begin{lemma}[Convergence]
	\label{lem:convergence}
	Let the assumptions \ref{as:as1}-\ref{as:as5} be satisfied. Then, there exists a constant $C > 0$ such that the following relations hold true for any $t \in (0, T]$
	\begin{align*}
		(i) \q & \norm{\ovl{j}_n(t) - j(t)}_{\Hs^{-1/2}(\Gm)} \le C \tau, \\
		& \norm{\ovl{\vb}_n(t) - \vb(t)}_{\CCs(\ovl{\Omega})} \le C \tau, \\
		(ii) \q & \lim\limits_{n \to \infty} \norm{\ovl{\kappa}_n(t) - \kappa(t)}_{\Ls^2(\Om)} = 0,\\
		& \lim\limits_{n \to \infty} \norm{\ovl{\sm}_n(t) - \sm(t)}_{\Ls^2(\Om)} = 0,\\
		& \lim\limits_{n \to \infty} \norm{\ovl{\alpha}_n(t) - \alpha(t)}_{\Ls^2(\Om)} = 0, \\
		& \lim\limits_{n \to \infty} \norm{\chi_{\widetilde{\Sm}_n(t)} - \chi_{\Sm(t)}}_{\Ls^2(\Om)} = 0.
	\end{align*}
\end{lemma}

\begin{proof}
	(i) \q For any $t \in (0, T]$, the Lipschitz continuity in time of the functions $j$ and $\vb$ gives us that
	\begin{align*}
		& \norm{\ovl{j}_n(t) - j(t)}_{\Hs^{-1/2}(\Gm)} \leqc \tau, \\
		& \norm{\ovl{\vb}_n(t) - \vb(t)}_{\CCs(\ovl{\Omega})} \leqc {\tau}.
	\end{align*}
	(ii)\q Thanks to the property of the mapping $\Phib$, it holds for any $t \in (0, T]$ that
	\begin{align*}
		\lim\limits_{n \to \infty} \norm{\ovl{\kappa}_n(t) - \kappa(t)}^2_{\Ls^2(\Om)} 
		& = \lim\limits_{n \to \infty} \norm{\ovl{\kappa}_n(t) - \kappa(t)}^2_{\Ls^2(\Sm(t))} + \lim\limits_{n \to \infty} \norm{\ovl{\kappa}_n(t) - \kappa(t)}^2_{\Ls^2(\Xi(t))} \\
		& = (\kappa_{\Xi} - \kappa_{\Sm})^2 \lim\limits_{n \to \infty} \paren{ \abs{\widetilde{\Sm}_n(t) \cup \Sm(t)} - \abs{\widetilde{\Sm}_n(t) \cap \Sm(t)}} \stackrel{\ref{as:as3}}{=} 0.
	\end{align*}
	The remaining limit transitions can be obtained by the same reasoning, which completes the proof.
\end{proof}

In the next two theorems, we prove the convergence of Rothe's functions to the solution of the variational system \eqref{eq:vf_phi}-\eqref{eq:vf_u}.

\begin{theorem}[Existence of $\phi$ and $\Ab$]\label{thm:existence_phi_A}%
	Let the assumptions \ref{as:as1}-\ref{as:as5} be fulfilled. Then, there exists a unique solution $(\phi, \Ab)$ to the variational problems \eqref{eq:vf_phi}-\eqref{eq:vf_A}, which satisfies $\phi \in \Lips([0, T], \Zs), \Ab \in \Cs([0, T], \Wb_0)$ with $\pa_t \Ab \in \Ls^2((0,T), \Wb_0)$ and $\Ab(0) = \tilde{\Ab}_0$ a.e. in $\Theta(0)$. Moreover, the following convergences hold true
	\begin{align}
		& \ovl{\phi}_n \to \phi && \text{in} \q \Ls^2((0, T), \Zs), \label{thm:existence_phi_A:eq_conv1} \\
		& \ovl{\Ab}_n \to \Ab, \q \Ab_n \to \Ab && \text{in} \q \Ls^2((0, T), \Wb_0), \label{thm:existence_phi_A:eq_conv3} \\
		& \ovl{\sm}_n \pa_t \Ab_n \sra \sm \pa_t \Ab && \text{in} \q \Ls^2((0, T), \LLs^2(\Om)), \label{thm:existence_phi_A:eq_conv4} \\
		& \sqrt{\ovl{\sm}_n} \pa_t \Ab_n \to \sqrt{\sm} \pa_t \Ab && \text{in} \q \Ls^2((0, T), \LLs^2(\Om)).\label{thm:existence_phi_A:eq_conv5}
	\end{align}
\end{theorem}

\begin{proof}
	The existence of a solution $(\phi, \Ab)$ to the variational system \eqref{eq:vf_phi}-\eqref{eq:vf_A} has already been shown in \cite[Theorems~5.1, 6.1 and 6.2]{LSV2021b}, where $\phi \in \Lips([0, T], \Zs), \Ab \in \Ls^{\infty}((0, T), \Wb_0)$ with $\sm \pa_t \Ab \in \Ls^2((0,T), \LLs^2(\Om))$ and $\Ab(0) = \tilde{\Ab}_0$ a.e. in $\Theta(0)$. Moreover, the convergences \eqref{thm:existence_phi_A:eq_conv1}-\eqref{thm:existence_phi_A:eq_conv4} have also been proved. Therefore, we omit these proofs.
	
	Next, the uniform boundedness of the sequence $\brac{\pa_t \Ab_n}$ in $\Ls^2((0, T), \Wb_0)$ (cf. \Cref{lem:est_A}) and the reflexivity of that space ensure the existence of a subsequence $\brac{\pa_t \Ab_{n_k}} \sst \brac{\pa_t \Ab_n}$ such that
	\begin{align}
		\label{eq:conv_pdtA}
		& \pa_t \Ab_{n_k} \sra \fb && \text{in} \q \Ls^2((0, T), \Wb_0).
	\end{align}
	By means of \cite[Lemma~1.3.6]{Kacur1985}, we get that $\fb = \pa_t \Ab$ in $\Ls^2((0, T), \Wb_0)$, and hence $\Ab \in \Cs([0, T], \Wb_0)$. Moreover, the equation \eqref{eq:conv_pdtA} is still valid for the whole sequence $\brac{\pa_t\Ab_n}$ due to the uniqueness of a weak solution $\Ab$, see \Cref{thm:uniqueness}.
	
	Finally, we show that the convergence \eqref{thm:existence_phi_A:eq_conv5} also holds true. Because the electrical conductivity $\sm$ vanishes on the air, the limit transition \eqref{thm:existence_phi_A:eq_conv4} immediately implies that
	\begin{align}
		& \pa_t \Ab_n \sra \pa_t \Ab && \text{in} \q \Ls^2((0, T), \LLs^2(\Pi)), \label{eq:conv_1} \\
		& \chi_{\widetilde{\Sm}_n} \pa_t \Ab_n \sra \chi_{\Sm} \pa_t \Ab && \text{in} \q \Ls^2((0, T), \LLs^2(\Om)). \label{eq:conv_2}
	\end{align}
	Hence, we can conclude that
	\begin{align}
		\label{eq:weak_conv}
		& \sqrt{\ovl{\sm}_n} \pa_t \Ab_n \sra \sqrt{\sm} \pa_t \Ab && \text{in} \q \Ls^2((0, T), \LLs^2(\Om)).
	\end{align}
	Now, setting $\vphib = \pa_t \Ab_n(t) \in \Wb_0$ in \eqref{eq:Rothe_A} and then integrating over the time range $(0, \eta) \sst (0, T)$ gives that
	\begin{align*}
		\int\limits_0^{\eta} \norm{\sqrt{\ovl{\sm}_n(t)} \ \pa_t \Ab_n(t)}^2_{\LLs^2(\Om)} \dt 
		& = \int\limits_0^{\eta} \paren{\ovl{\sm}_n(t)\ \pa_t \Ab_n(t), \pa_t \Ab_n(t)}_{\widetilde{\Theta}_n(t)} \dt \\
		& = - \mu^{-1}_0 \int\limits_0^{\eta} \paren{\curl{\ovl{\Ab}_n(t)}, \curl{\pa_t \Ab_n(t)}}_{\Om} \dt - \sm_{\Pi} \int\limits_0^{\eta} \paren{\nabla \ovl{\phi}_n(t), \pa_t \Ab_n(t)}_{\Pi} \dt \\
		& \qqq + \sm_{\Sm} \int\limits_0^{\eta} \paren{\ovl{\vb}_n(t) \times \left(\curl{\ovl{\Ab}_n(t)}\right), \pa_t \Ab_n(t)}_{\widetilde{\Sm}_n(t)} \dt.
	\end{align*}
	By virtue of the limit transitions \eqref{thm:existence_phi_A:eq_conv1}-\eqref{thm:existence_phi_A:eq_conv4} and \eqref{eq:conv_pdtA}-\eqref{eq:conv_2}, we are able to pass to the limit for $n \to \infty$ as follows
	\begin{align*}
		\lim\limits_{n \to \infty} \int\limits_0^{\eta} \norm{\sqrt{\ovl{\sm}_n(t)} \ \pa_t \Ab_n(t)}^2_{\LLs^2(\Om)} \dt 
		& = - \mu^{-1}_0 \int\limits_0^{\eta} \paren{\curl{\Ab(t)}, \curl{\pa_t \Ab(t)}}_{\Om} \dt - \sm_{\Pi} \int\limits_0^{\eta} \paren{\nabla \phi(t), \pa_t \Ab(t)}_{\Pi} \dt \\
		& \qqqq + \sm_{\Sm} \int\limits_0^{\eta} \paren{\vb(t) \times \left(\curl{\Ab(t)}\right), \pa_t \Ab(t)}_{\Sm(t)} \dt \\
		& \stackrel{\eqref{eq:vf_A}}{=} \int\limits_0^{\eta} \paren{\sm(t) \pa_t \Ab(t), \pa_t \Ab(t)}_{\Theta(t)} \dt = \int\limits_0^{\eta} \norm{\sqrt{\sm(t)} \ \pa_t \Ab(t)}^2_{\LLs^2(\Om)} \dt.
	\end{align*}
	This relation together with the weak convergence \eqref{eq:weak_conv} leads us to the strong convergence \eqref{thm:existence_phi_A:eq_conv5}.
\end{proof}

\begin{theorem}[Existence of $u$]
	\label{thm:existence_u}
	Let the assumptions \ref{as:as1}-\ref{as:as5} be fulfilled. Then, there exists a unique function {$u \in \Ys_\alpha \cap \Ls^{\infty}((0, T), \Ls^2(\Om))$} such that the triplet $\paren{\phi, \Ab, u}$ solves the variational problem \eqref{eq:vf_u} and $u(0) = \tilde{u}_0$ a.e. in $\Om$. In addition, the following convergences hold true
	\begin{align*}
		& \ovl{u}_n \sra u, \qq \unl{u}_n \sra u && \text{in} \q \Ls^2((0, T), \Hs^1(\Om)), \\
		& {\ovl{\alpha}_n \pa_t u_n \sra \alpha \pa_t u} && \text{in} \q \Ls^2((0, T), [\Hs^1(\Om)]^\prime).
	\end{align*}
\end{theorem}

\begin{proof}
	First of all, we introduce some auxiliary identities that are useful for further analysis.
	For any time $\eta \in (t_{l-1}, t_l]$, with $l = 1, 2, \ldots, n$, one can easily see that
	\begin{align*}
		& \int\limits_0^{\eta} \paren{\ovl{\alpha}_n(t) \pa_t u_n(t), w}_{\Om} \dt\\
		& = \sum\limits_{i=1}^l \paren{\alpha_i (u_i - u_{i-1}), w}_{\Om}  - \int\limits_{\eta}^{t_l} \paren{\ovl{\alpha}_n(t) \pa_t u_n(t), w}_{\Om} \dt \\
		& = \sum\limits_{i=1}^l \paren{\alpha_i u_i - \alpha_{i-1} u_{i-1}, w}_{\Om} - \sum\limits_{i=1}^l \paren{(\alpha_i - \alpha_{i-1}) u_{i-1}, w}_{\Om} - \int\limits_{\eta}^{t_l} \paren{\ovl{\alpha}_n(t) \pa_t u_n(t), w}_{\Om} \dt.
	\end{align*}
	We split the second term on the RHS over the subdomains as follows
	\begin{align*}
		\sum\limits_{i=1}^l \paren{(\alpha_i - \alpha_{i-1})u_{i-1}, w}_\Om
		& = \sum\limits_{i=1}^l \paren{\alpha_i u_{i-1}, w}_\Om - \sum\limits_{i=1}^l \paren{\alpha_{i-1} u_{i-1}, w}_\Om \\
		& = \sum\limits_{i=1}^l \paren{\alpha_i u_{i-1}, w}_{\Sigma_i \cup \Xi_i \cup \Pi} - \sum\limits_{i=1}^l \paren{\alpha_{i-1} u_{i-1}, w}_{\Sigma_{i-1} \cup \Xi_{i-1} \cup \Pi}.
	\end{align*}
	Then, the RTT allows us to rewrite that
	\[
	\paren{\alpha_i u_{i-1}, w}_{\Sm_i} - \paren{\alpha_{i-1} u_{i-1}, w}_{\Sm_{i-1}} 
	= \alpha_\Sm \int\limits_{t_{i-1}}^{t_i} \dfrac{\di}{\dt} \int\limits_{\Sm(t)} u_{i-1} w \dx \dt \stackrel{\eqref{eq:Reynolds2}}{=} \alpha_\Sm \int\limits_{t_{i-1}}^{t_i} \int\limits_{\Sm(t)} \dive{\paren{u_{i-1}\, w \, \vb(t)}} \dx \dt.    
	\]
	A similar identity can be obtained for the integrals over the air subdomains $\Xi_i$ and $\Xi_{i-1}$, while the corresponding terms disappear on the fixed coil $\Pi$. Therefore, we arrive at
	\begin{align}
		\label{eq:pdt}
		\int\limits_0^{\eta} \paren{\ovl{\alpha}_n(t) \pa_t u_n(t), w}_{\Om} \dt & = \paren{\ovl{\alpha}_n(\eta) \ovl{u}_n(\eta), w}_{\Om} - \paren{\alpha(0) \tilde{u}_0, w}_{\Om} - \int\limits_0^{\eta} \int\limits_\Om \alpha(t) \dive{\paren{\unl{u}_n(t) \, w \, \vb(t)}} \dx \dt \notag \\
		& \qqq - \int\limits_{\eta}^{\ovl{\eta}_n} \int\limits_\Om \alpha(t) \dive{\paren{\unl{u}_n(t) \, w \, \vb(t)}} \dx \dt - \int\limits_{\eta}^{\ovl{\eta}_n} \paren{\ovl{\alpha}_n(t) \pa_t u_n(t), w}_{\Om} \dt.
	\end{align}
	Next, the uniform boundedness of $\brac{\ovl{u}_n}$ from \Cref{lem:est_u} together with the reflexivity of $\Ls^2((0, T), \Hs^1(\Om))$ ensures the existence of a subsequence $\brac{\ovl{u}_{n_k}} \sst \brac{\ovl{u}_n}$ (denoted further by the same index with the original sequence) such that
	\begin{align*}
		& \ovl{u}_n \sra u && \text{in} \q \Ls^2((0, T), \Hs^1(\Om)).
	\end{align*}
	Moreover, by a similar argument, we have the existence of $y \in \Ls^2((0, T), \Hs^1(\Om))$ and $z \in \Ls^2((0, T), [\Hs^1(\Om)]^\prime)$ such that 
	\begin{align*}
		& \unl{u}_n \sra y && \text{in} \q \Ls^2((0, T), \Hs^1(\Om)), \\
		& \ovl{\alpha}_n \pa_t u_n \sra z && \text{in} \q \Ls^2((0, T), [\Hs^1(\Om)]^\prime). 
	\end{align*}
	Because of the a priori estimate \eqref{eq:est_u} for $u_i$, the following relation between $\{\ovl{u}_n\}$ and $\{\unl{u}_n\}$ holds true
	\[
	0 \le \lim\limits_{n \to \infty} \norm{\ovl{u}_n - \unl{u}_n}^2_{\Ls^2((0, T), \Ls^2(\Om))} = \lim\limits_{n \to \infty} \sum\limits_{i=1}^n \norm{u_i - u_{i-1}}^2_{\Ls^2(\Om)} \tau \stackrel{\eqref{eq:est_u}}{\leqc} \lim\limits_{n \to \infty} \tau = 0,
	\]
	which implies that $y = u$ in $\Ls^2((0, T), \Hs^1(\Om))$.
	In addition, we have that $u \in \Ls^{\infty}((0, T), \Ls^2(\Om))$ thanks to the relation
	\[
	\max\limits_{t \in (0, T]} \norm{\ovl{u}_n(t)}_{\Ls^2(\Om)} \stackrel{\eqref{eq:est_u}}{\leqc} 1.
	\]
	In the following, we prove that the function $u$ is the solution to the variational problem \eqref{eq:vf_u}.
	To this end, we integrate the equation \eqref{eq:Rothe_u} over the time interval $(0, \eta) \sst (0, T)$, and then integrate the result over $(0, \xi) \sst (0, T)$. By means of the identity \eqref{eq:pdt}, we get that
	\begin{equation}
		\label{eq:pp}
		\begin{split}
			& \int\limits_0^{\xi} \paren{\ovl{\alpha}_n(\eta) \ovl{u}_n(\eta), w}_{\Om} \di \eta - \xi \paren{\alpha(0) \tilde{u}_0, w}_{\Om} - \int\limits_0^{\xi} \int\limits_0^{\eta} \int\limits_\Om \alpha(t) \dive{\paren{\unl{u}_n(t) \, w \, \vb(t)}} \dx \dt \di \eta \\
			& \q - \int\limits_0^{\xi} \int\limits_{\eta}^{\ovl{\eta}_n} \int\limits_\Om \alpha(t) \dive{\paren{\unl{u}_n(t) \, w \, \vb(t)}} \dx \dt \di \eta - \int\limits_0^{\xi} \int\limits_{\eta}^{\ovl{\eta}_n} \paren{\ovl{\alpha}_n(t) \pa_t u_n(t), w}_{\Om} \dt \di \eta  \\
			& \q + \int\limits_0^{\xi} \int\limits_0^{\eta} \paren{\ovl{\alpha}_n(t) \ovl{\vb}_n(t) \cdot \nabla \ovl{u}_n(t), w}_{\Om} \dt \di \eta + \int\limits_0^{\xi} \int\limits_0^{\eta} \paren{\ovl{\kappa}_n(t) \nabla \ovl{u}_n(t), \nabla w}_{\Omega} \dt \di \eta \\
			& \q = \int\limits_0^{\xi} \int\limits_0^{\eta} \paren{\RR_r\left(\ovl{Q}_n(t)\right), w}_{\widetilde{\Theta}_n(t)} \dt \di \eta.
		\end{split}
	\end{equation}
	Let us invoke \Cref{lem:est_u} to obtain that
	\begin{align*}
		& \lim\limits_{n \to \infty} \abs{ \int\limits_0^{\xi} \int\limits_{\eta}^{\ovl{\eta}_n} \int\limits_\Om \alpha(t) \dive{\paren{\unl{u}_n(t) \, w \, \vb(t)}} \dx \dt \di \eta} \ \leqc \lim\limits_{n \to \infty} \sum\limits_{i=1}^n \norm{u_{i-1}}_{\Hs^1(\Om)} \norm{w}_{\Hs^1(\Om)} \tau^2  \stackrel{\eqref{eq:est_u}}{\leqc} \lim\limits_{n \to \infty} \tau = 0, \\
		& \lim\limits_{n \to \infty} \abs{ \int\limits_0^{\xi} \int\limits_{\eta}^{\ovl{\eta}_n} \paren{\ovl{\alpha}_n(t) \pa_t u_n(t), w}_{\Om} \dt \di \eta \ } \leqc \lim\limits_{n \to \infty} \sum\limits_{i=1}^n \norm{u_i - u_{i-1}}_{\Ls^2(\Om)} \norm{w}_{\Ls^2(\Om)} \tau \stackrel{\eqref{eq:est_u}}{\leqc} \lim\limits_{n \to \infty} \sqrt{\tau} = 0.
	\end{align*}
	Using the convergences in \Cref{lem:convergence} and the Lebesgue dominated convergence theorem, we are able to show that
	\begin{align*}
		& \lim\limits_{n \to \infty} \int\limits_0^{\xi} \paren{\ovl{\alpha}_n(\eta) \ovl{u}_n(\eta), w}_{\Om} \di \eta = \int\limits_0^{\xi} \paren{\alpha(\eta) u(\eta), w}_{\Om} \di \eta, \\
		& \lim\limits_{n \to \infty} \int\limits_0^{\xi} \int\limits_0^{\eta} \int\limits_\Om \alpha(t) \dive{\paren{\unl{u}_n(t) \, w \, \vb(t)}} \dx \dt \di \eta = \int\limits_0^{\xi} \int\limits_0^{\eta} \int\limits_\Om \alpha(t) \dive{\paren{u(t) \, w \, \vb(t)}} \dx \dt \di \eta, \\
		& \lim\limits_{n \to \infty} \int\limits_0^{\xi} \int\limits_0^{\eta} \paren{\ovl{\alpha}_n(t) \ovl{\vb}_n(t) \cdot \nabla \ovl{u}_n(t), w}_{\Om} \dt \di \eta = \int\limits_0^{\xi} \int\limits_0^{\eta} \paren{\alpha(t) \vb(t) \cdot \nabla u(t), w}_{\Om} \dt \di \eta, \\
		& \lim\limits_{n \to \infty} \int\limits_0^{\xi} \int\limits_0^{\eta} \paren{\ovl{\kappa}_n(t) \nabla \ovl{u}_n(t), \nabla w}_{\Omega} \dt \di \eta = \int\limits_0^{\xi} \int\limits_0^{\eta} \paren{\kappa(t) \nabla u(t), \nabla w}_{\Omega} \dt \di \eta.
	\end{align*}
	Now, we address the convergence of the remaining term concerning the discrete Joule heat source $\ovl{Q}_n$. 
	Let us denote 
	\begin{align*}
		\qb & = \sqrt{\sm} \paren{\pa_t \Ab + \chi_{\Pi} \nabla \phi - \vb \times \left(\curl{\Ab}\right)}, \\
		\ovl{\qb}_n & = \sqrt{\ovl{\sm}_n} \paren{\pa_t \Ab_n + \chi_{\Pi} \nabla \ovl{\phi}_n - \ovl{\vb}_n \times \left(\curl{\ovl{\Ab}_n}\right)}.
	\end{align*}
	One can immediately see from \Cref{thm:existence_phi_A} that
	\begin{align*}
		& \ovl{\qb}_n \to \qb && \text{in} \q \Ls^2((0, T), \LLs^2(\Om)).
	\end{align*}
	In addition, we introduce the following inequality, which is valid for non-negative numbers $a, b$ and $c$ 
	\begin{equation}
		\label{eq:useful}
		\abs{\min(a, b) - \min(a, c) \ } \le 2 \sqrt{a} \ \abs{\sqrt{b} - \sqrt{c} \ }.
	\end{equation}
	This inequality can be easily proven by considering all possible cases. Using this inequality and the definition of the cut-off function $\RR_r$, we have that
	\[
	\abs{\RR_r\left(\ovl{Q}_n\right) - \RR_r(Q)} = \abs{\min\left(r, \abs{\ovl{\qb}_n}^2\right) - \min\left(r, \abs{\qb}^2\right) \ } \stackrel{\eqref{eq:useful}}{\le} 2 \sqrt{r} \ \abs{\ {\abs{\ovl{\qb}_n} - \abs{\qb}}\ } \le 2 \sqrt{r} \ \abs{{\ovl{\qb}_n - \qb}}.
	\]
	Therefore, for each $w \in \Hs^1(\Om)$, we can deduce that
	\begin{align*}
	0 &\le \lim\limits_{n \to \infty} \abs{\int\limits_0^{\xi} \int\limits_0^{\eta} \paren{\RR_r\left(\ovl{Q}_n(t)\right) - \RR_r(Q(t)), w}_{\Om} \dt \di \eta \ } \\
	& \leqc \lim\limits_{n \to \infty} \norm{w}_{\Ls^2(\Om)} \left(\int\limits_0^{\xi} \int\limits_0^{\eta} \norm{\ovl{\qb}_n(t) - \qb(t)}^2_{\LLs^2(\Om)} \dt \di \eta \right)^{1/2} = 0.
	\end{align*}
	Collecting all limit transitions above, we are able
	to pass to the limit for $n \to \infty$ in \eqref{eq:pp} to arrive at
	\begin{align}
		\label{eq:cont}
		& \int\limits_0^{\xi} \paren{\alpha(\eta) u(\eta), w}_{\Om} \di \eta - \xi \paren{\alpha(0) \tilde u_0, w}_{\Om} - \int\limits_0^{\xi} \int\limits_0^{\eta} \int\limits_\Om \alpha(t) \dive{\paren{u(t) \, w \, \vb(t)}} \dx \dt \di \eta \notag \\
		& \q + \int\limits_0^{\xi} \int\limits_0^{\eta} \paren{\alpha(t) \vb(t) \cdot \nabla u(t), w}_{\Om} \dt \di \eta + \int\limits_0^{\xi} \int\limits_0^{\eta} \paren{\kappa(t) \nabla u(t), \nabla w}_{\Omega} \dt \di \eta = \int\limits_0^{\xi} \int\limits_0^{\eta} \paren{\RR_r(Q(t)), w}_{\Theta(t)} \dt \di \eta.
	\end{align}
	Differentiating \eqref{eq:cont} w.r.t. $\xi$ gives that
	\begin{equation}
		\label{eq:cont_diff}
		\begin{aligned}
			& \paren{\alpha(\xi) u(\xi), w}_{\Om} -  \paren{\alpha(0) \tilde u_0, w}_{\Om} - \int\limits_0^{\xi} \int\limits_\Om \alpha(t) \dive{\paren{u(t) \, w \, \vb(t)}} \dx \dt \\
			& \qq + \int\limits_0^{\xi} \paren{\alpha(t) \vb(t) \cdot \nabla u(t), w}_{\Om} \dt + \int\limits_0^{\xi} \paren{\kappa(t) \nabla u(t), \nabla w}_{\Omega} \dt = \int\limits_0^{\xi} \paren{\RR_r(Q(t)), w}_{\Theta(t)} \dt,
		\end{aligned}
	\end{equation}
	which implies the absolute continuity of ${\mathcal F}_w(t) := (\alpha(t) u(t), w)_\Om$ on $[0, T]$ for all $w \in \Hs^1(\Om)$. In addition, by differentiating \eqref{eq:cont_diff} w.r.t. $\xi$, it is clear that there exists a function $g \in \Ls^1((0, T), [\Hs^1(\Om)]^\prime)$ satisfying the condition \eqref{eq:defder_g} in \Cref{def:derivative}, i.e., $\alpha \pa_t u \in \Ls^1((0, T), [\Hs^1(\Om)]^\prime)$.
	Hence, we obtain from \eqref{eq:pdt} that 
	\begin{align*}
		\int\limits_0^{\xi} \int\limits_0^{\eta} \inprod{z(t), w}_{1, \Om} \dt \di\eta
		& = \lim_{n \to \infty} \int\limits_0^{\xi} \int\limits_0^{\eta} \paren{\ovl{\alpha}_n(t) \pa_t u_n(t), w}_\Om \dt \di \eta \\
		& = \int\limits_0^{\xi} \paren{\alpha(\eta) u(\eta), w}_{\Om} \di \eta - \xi \paren{\alpha(0) \tilde u_0, w}_{\Om} - \int\limits_0^{\xi} \int\limits_0^{\eta} \int\limits_\Om \alpha(t) \dive{\paren{u(t) \, w \, \vb(t)}} \dx \dt \di \eta \\
		& = \int\limits_0^{\xi} \int\limits_0^{\eta} \dfrac{\di}{\dt} \paren{\alpha(t) u(t), w}_{\Om} \dt \di \eta - \int\limits_0^{\xi} \int\limits_0^{\eta} \int\limits_\Om \alpha(t) \dive{\paren{u(t) \, w \, \vb(t)}} \dx \dt \di \eta \\
		& = \int\limits_0^{\xi} \int\limits_0^{\eta} \inprod{(\alpha \partial_t u)(t), w}_{1, \Om} \dt \di \eta, 
	\end{align*}
	i.e., $z=\alpha \partial_t u \in \Ls^2((0, T), [\Hs^1(\Om)]^\prime).$
	Next, differentiating \eqref{eq:cont_diff} w.r.t. $\xi$ gives us back the variational problem \eqref{eq:vf_u}, which means that $\phi, \Ab$ and $u$ solve the problem \eqref{eq:vf_u}. 
	Finally, we show that the initial condition of $u$ is satisfied. Multiplying the equation \eqref{eq:vf_u} by $\gamma \in \Cs^\infty([0,T])$ satisfying $\gamma(0) = 1$ and $\gamma(T)= 0$, then integrating over the time range $(0, T)$ gives us
	\begin{multline*}
	\int\limits_0^T \gamma(t) \inprod{(\alpha \pa_t u)(t), w}_{1, \Om} \dt + \int\limits_0^T \gamma(t) \paren{\alpha(t) \vb(t) \cdot \nabla u(t), w}_{\Om} \dt \\
	 + \int\limits_0^T \gamma(t) \paren{\kappa(t) \nabla u(t), \nabla w}_{\Om} \dt
	= \int\limits_0^T \gamma(t) \paren{\RR_r(Q(t)), w}_{\Theta(t)} \dt.    
	\end{multline*}
	The first term can be rewritten using the definition \eqref{def:alpha_pdt_u}, leading us to that
	\begin{multline*}
		\paren{\alpha(0) u(0), w}_{\Om} 
		 = - \int\limits_0^T \gamma^\prime(t) \paren{\alpha(t) u(t), w}_{\Om} \dt \\ - \int\limits_0^T \int\limits_\Om \gamma(t)\, \alpha(t) \dive{(u(t) \, w \, \vb(t))} \dx \dt 
		  + \int\limits_0^T \gamma(t) \paren{\alpha(t) \vb(t) \cdot \nabla u(t), w}_{\Om} \dt \\
		   + \int\limits_0^T \gamma(t)  \paren{\kappa(t) \nabla u(t), \nabla w}_{\Om} \dt - \int\limits_0^T \gamma(t) \paren{\RR_r(Q(t)), w}_{\Theta(t)} \dt.
	\end{multline*}
	We repeat the process above when considering \eqref{eq:Rothe_u}, and then pass to the limit $n \to \infty$ to have that 
	\begin{multline*}
		\paren{\alpha(0) \tilde{u}_0, w}_{\Om} 
		 = - \int\limits_0^T \gamma^\prime(t) \paren{\alpha(t) u(t), w}_{\Om} \dt \\
		 - \int\limits_0^T \int\limits_\Om \gamma(t)\, \alpha(t) \dive{(u(t) \, w \, \vb(t))} \dx \dt 
		 + \int\limits_0^T \gamma(t) \paren{\alpha(t) \vb(t) \cdot \nabla u(t), w}_{\Om} \dt\\
		  + \int\limits_0^T \gamma(t) \paren{\kappa(t) \nabla u(t), \nabla w}_{\Om} \dt - \int\limits_0^T \gamma(t) \paren{\RR_r(Q(t)), w}_{\Theta(t)} \dt.
	\end{multline*}
	Therefore, $\paren{\alpha(0) u(0) - \alpha(0) \tilde{u}_0, w}_{\Om} = 0$ for all $w \in \Hs^1(\Om)$, which implies that $u(0) = \tilde{u}_0$ a.e. in $\Om$. By taking into account the uniqueness of the solution $u$ from \Cref{thm:uniqueness}, we note that the convergences in \Cref{thm:existence_u} are not only valid for a subsequence, but also for the original sequence. We have accomplished the proof.
\end{proof}
\section{Numerical results}
\label{sec:numerial_results}

We perform some numerical tests in this section to support our theoretical results. Time and space discretization schemes for electromagnetic problems involving a moving non-magnetic conductor have been thoroughly studied in our previous works (cf. \cite{LSV2021b,LSV2022c}). Therefore, we are now focusing on the performance of the discretization scheme for the heat problem with a moving domain. In \Cref{subsec:experiments}, two numerical experiments describing the heat transfer process in two-dimensional (2D) rotating disks are investigated. Afterwards, in \Cref{subsec:simulation}, we repeat the simulation of an induction heating system performed in \cite{CGS2017}, with a moving workpiece.

The variational problems are numerically solved using the FEM, and the discretization scheme is implemented with the aid of the finite-element software package FreeFEM \cite{Hecht2012}. The implementation of the variational problems \eqref{eq:TD_phi}-\eqref{eq:TD_A} follows the saddle-point formulations proposed in \cite{LSV2021b} and \cite{LSV2022c}. The first-order Lagrangian finite elements are used to spatially approximate the solution $u_i$ of the equation \eqref{eq:TD_u}. Since the velocity $\vb$ is known, it is not necessary to change the computational mesh, which requires a re-meshing procedure that would significantly increase the computational cost. Instead, the mesh is fixed during the whole time range, and a characteristic function tracks the moving workpiece  
\[
\chi_{\Sm}(\xb) = 
\begin{cases}   
	1 \qq  \text{if } \xb \in \Sm, \\
	0 \qq \text{otherwise}.
\end{cases}
\]

In order to estimate the order of convergence without knowing the exact solution, we define the following relative error between Rothe's solution $\ovl{u}_n$ obtained by the proposed numerical method and a reference solution $u_{ref}$
\[
\tilde{E}_u = \dfrac{\norm{\ovl{u}_n - u_{ref}}^2_{\Ls^2((0, T), \Hs^1(\Om))}}{\norm{u_{ref}}^2_{\Ls^2((0, T), \Hs^1(\Om))}}.
\]

In all test cases, we assume that the initial temperature is $\tilde{u}_0 = 298\mathrm{K} \ (\approx 25^\circ \mathrm{C})$ and the constant magnetic permeability of vacuum is $\mu_0 = 4\pi$E-7$\mathrm{H/m}$. The values of other material coefficients used in numerical tests are presented in \Cref{tab:coefficients}, cf. \cite{Serway1998, Matula1979, TPHK1970}.
\begin{table}[http]
	\centering
	\caption{\em Material coefficients used in the numerical tests.} \label{tab:coefficients}
	\begin{tabular}{ p{0.25\textwidth} p{0.05\textwidth} c c c c } \hline
		& & Unit & Air & Copper & Aluminium \\ \hline 
		Electrical conductivity     & $\sm$     & $\mathrm{MS/m}$         & -         & 59.6    & 35     \\ 
		Volumetric heat capacity    & $\alpha$   & $\mathrm{kJ/(m^3{\cdot}K)}$   & 1.192  & 3384   & 2422    \\
		Thermal conductivity        & $\kappa$  & $\mathrm{W/(m{\cdot}K)}$     & 0.02514   & 401       & 237       \\ \hline
	\end{tabular}
\end{table}

\subsection{Numerical experiments}
\label{subsec:experiments}

We perform two numerical experiments concerning the heat transfer process in 2D rotating disks with radius $r_1 = {0.2}\mathrm{m}$. The disks both consist of an aluminium circular area, with radius $r_2 = {0.1}\mathrm{m}$ and $r_2 = {0.05}\mathrm{m}$, respectively, and the complementary area filled by copper, see \Cref{fig:exp_domain}. The domains are rotating with velocity $\vb = {0.125}\pi \ (-y, x)^\transpose \mathrm{m/s}$ and are partitioned into {245598} and {251536} triangles, respectively. Instead of the Joule heating $Q$, the system is supplied with a heat source $f = {1} \mathrm{MW/m^3}$. The changes over time of the temperature distribution are visualized by the software package MEDIT \cite{Frey2001}, which are shown in~\Cref{fig:exp1,fig:exp2}.

\begin{figure}
	\centering
	\begin{subfigure}{.5\textwidth}
		\centering
		\begin{tikzpicture}
			\draw[fill=blue!50!, line width=0.25mm] (0,0) circle (2.5cm);
			\draw[fill=red!60] (0,0) circle (1.25cm);
			\draw[-latex, line width=0.25mm] (2.7,0) arc (0:30:2.7);
			\filldraw[black] (0,0) circle (2pt);
			\node at (2.8, 0.8) {$\vb$};
		\end{tikzpicture}
	\end{subfigure}%
	\begin{subfigure}{.5\textwidth}
		\centering
		\begin{tikzpicture}
			\draw[fill=blue!50!, line width=0.25mm] (0,0) circle (2.5cm);
			\draw[fill=red!60] (1.25,0) circle (0.625cm);
			\draw[-latex, line width=0.25mm] (2.7,0) arc (0:30:2.7);
			\filldraw[black] (0,0) circle (2pt);
			\node at (2.8, 0.8) {$\vb$};
		\end{tikzpicture}
	\end{subfigure}
	\caption{The circular domain of experiments consisting of an aluminium circular area (red) and the complementary area filled by copper (blue). The domains are rotating with velocity $\vb$. \textit{Left}: the first experiment with a concentric interior circle. \textit{Right}: the second experiment with an eccentric interior circle.}
	\label{fig:exp_domain}
\end{figure}
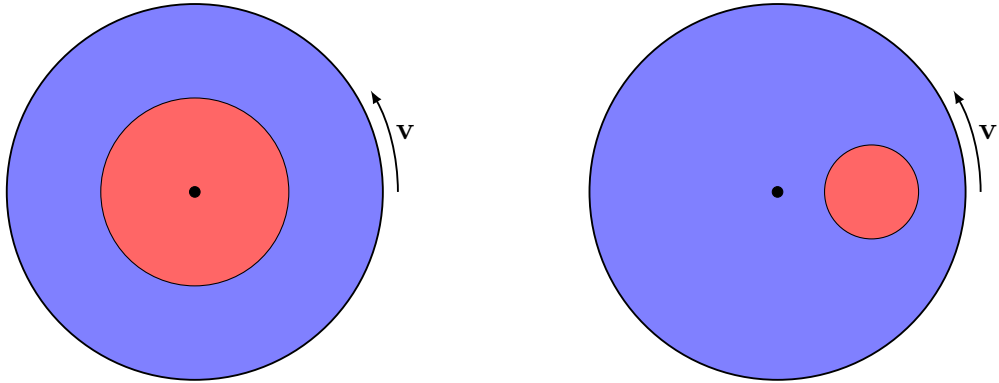

\begin{figure}
	\centering
	\begin{subfigure}{.33\textwidth}
		\centering
		\includegraphics[width=1\linewidth]{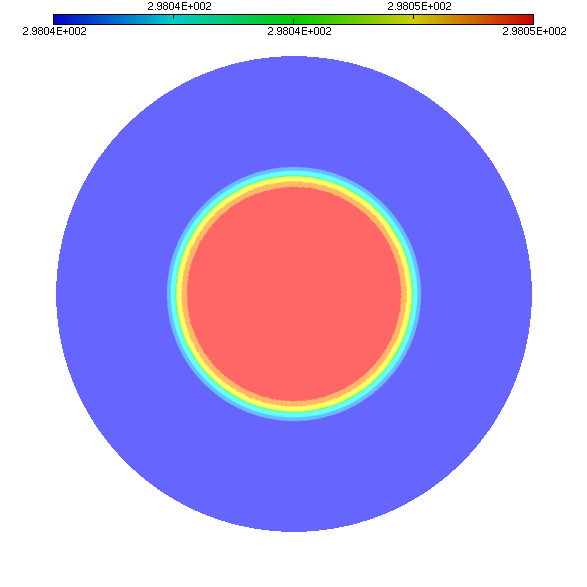}
	\end{subfigure}%
	\begin{subfigure}{.33\textwidth}
		\centering
		\includegraphics[width=1\linewidth]{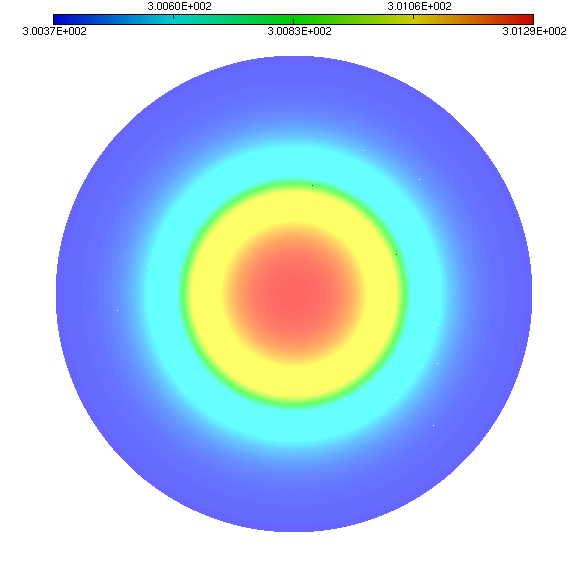}
	\end{subfigure}%
	\begin{subfigure}{.33\textwidth}
		\centering
		\includegraphics[width=1\linewidth]{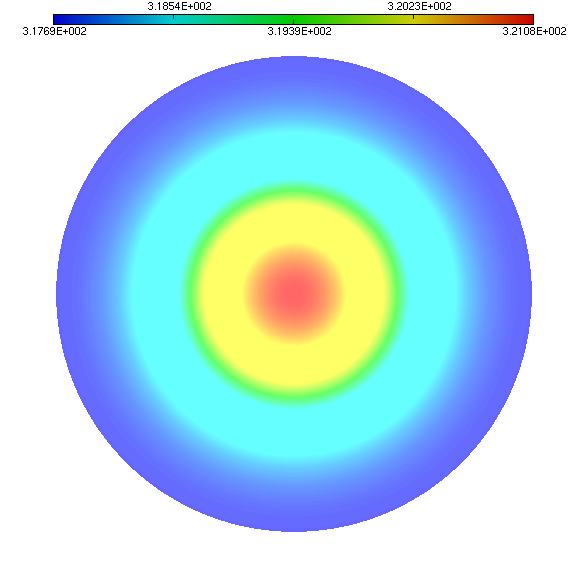}
	\end{subfigure}
	\caption{Temperature distribution of the first experiment at different time points. \textit{Left:} $t = 0.125 \mathrm{s}$. \textit{Middle:} $t = 8 \mathrm{s}$. \textit{Right:} $t = 64 \mathrm{s}$.}
	\label{fig:exp1}
\end{figure}

\begin{figure}
	\centering
	\begin{subfigure}{.33\textwidth}
		\centering
		\includegraphics[width=1\linewidth]{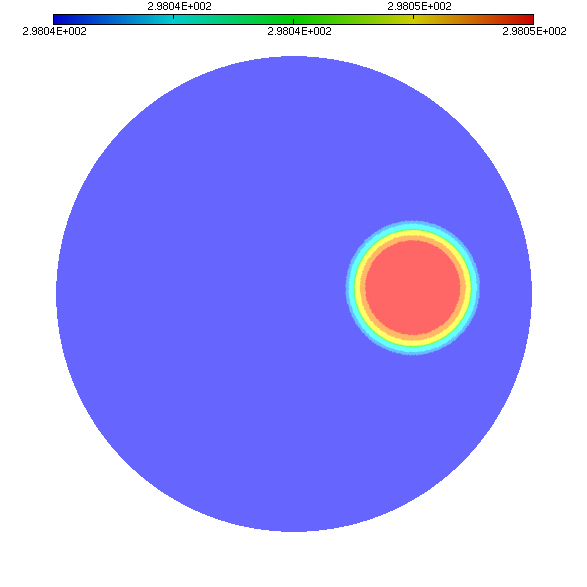}
	\end{subfigure}%
	\begin{subfigure}{.33\textwidth}
		\centering
		\includegraphics[width=1\linewidth]{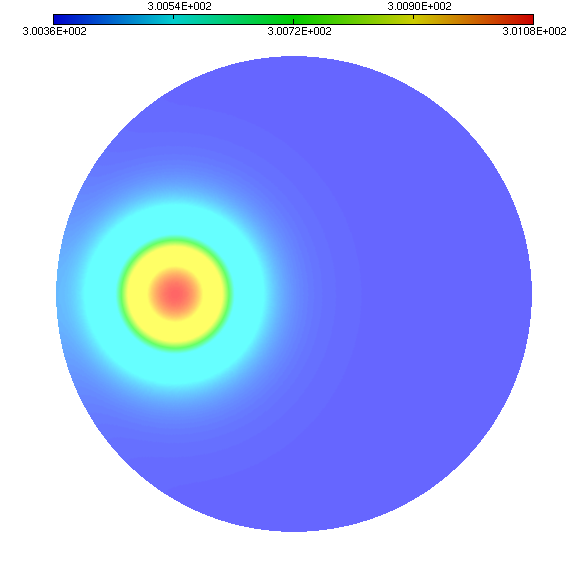}
	\end{subfigure}%
	\begin{subfigure}{.33\textwidth}
		\centering
		\includegraphics[width=1\linewidth]{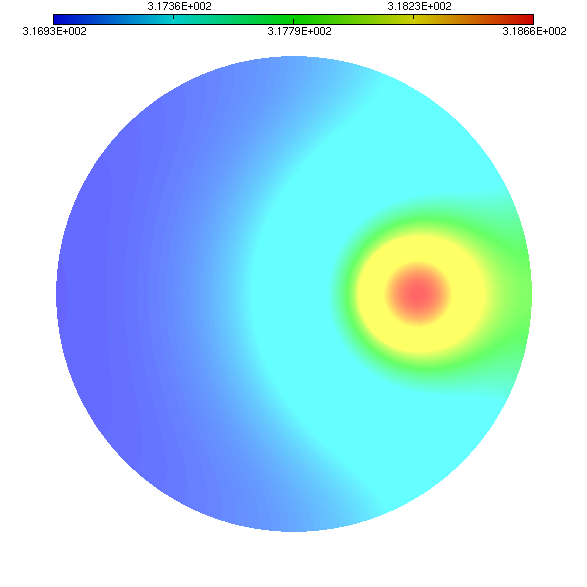}
	\end{subfigure}
	\caption{Temperature distribution of the second experiment at different time points. \textit{Left:} $t = 0.125 \mathrm{s}$. \textit{Middle:} $t = 8 \mathrm{s}$. \textit{Right:} $t = 64 \mathrm{s}$.}
	\label{fig:exp2}
\end{figure}

We verify the convergence of the temporal discretization scheme \eqref{eq:TD_u} in the time interval $(0, T)$ with $T = {64} \mathrm{s}$. The reference temperature $u_{ref}$ is the solution to \eqref{eq:TD_u} with time step $\tau = {2^{-8}}\mathrm{s}$, while larger time steps $\tau = 2^{-j}\mathrm{s}$, with $j = {2, 3, \ldots, 7}$, are used to compute the discrete solution $\ovl{u}_n$. Relative errors {$\tilde{E}_u$} w.r.t. time step $\tau$ are presented in~\Cref{fig:exp_convergence}. This figure shows the potential convergence rate $\OO(\tau)$ of the numerical scheme \eqref{eq:TD_u}.

\begin{figure}
	\centering
	\begin{subfigure}{.5\textwidth}
		\centering
		\includegraphics[width=1.05\textwidth]{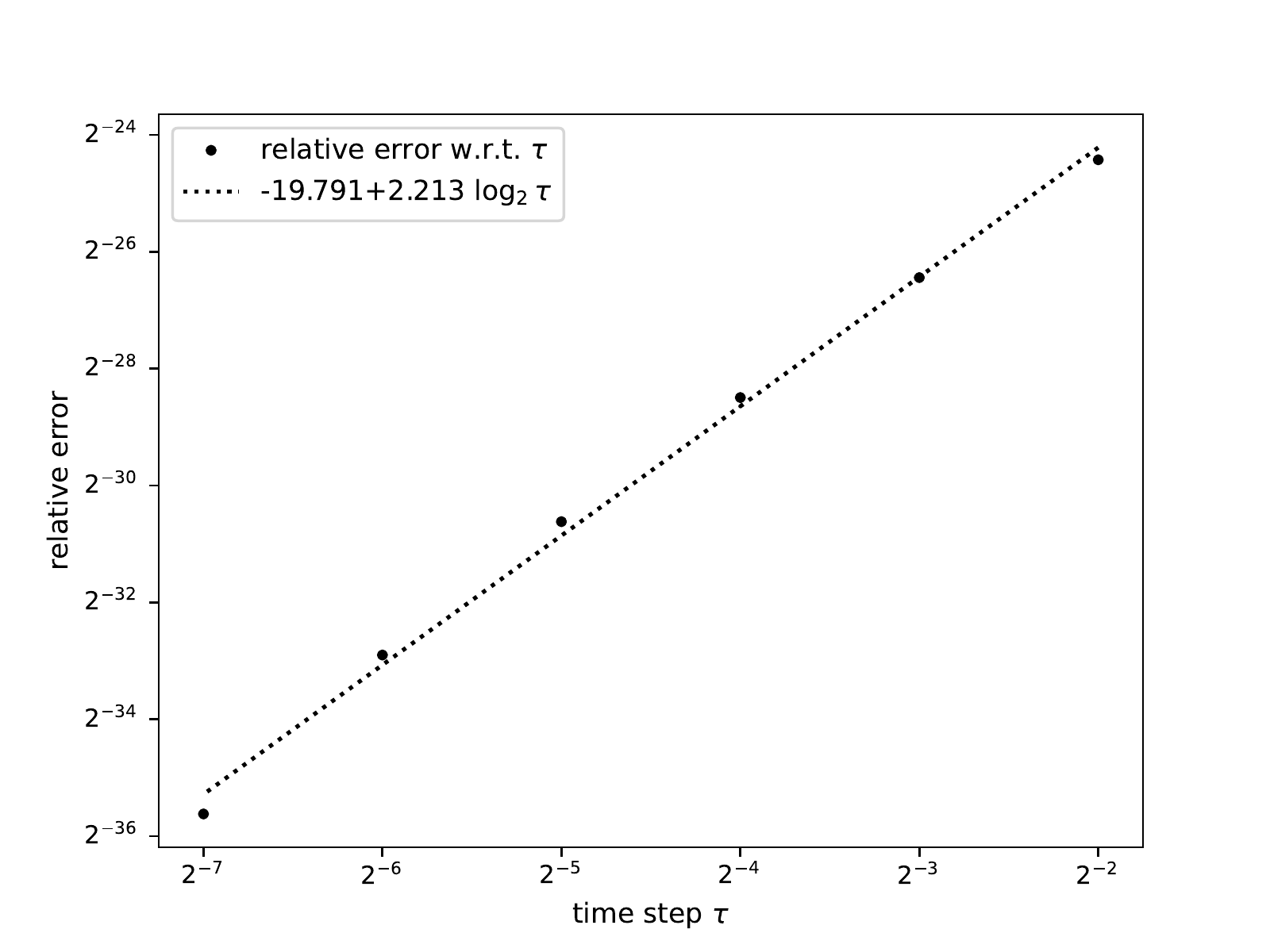}
	\end{subfigure}%
	\begin{subfigure}{.5\textwidth}
		\centering
		\includegraphics[width=1.05\textwidth]{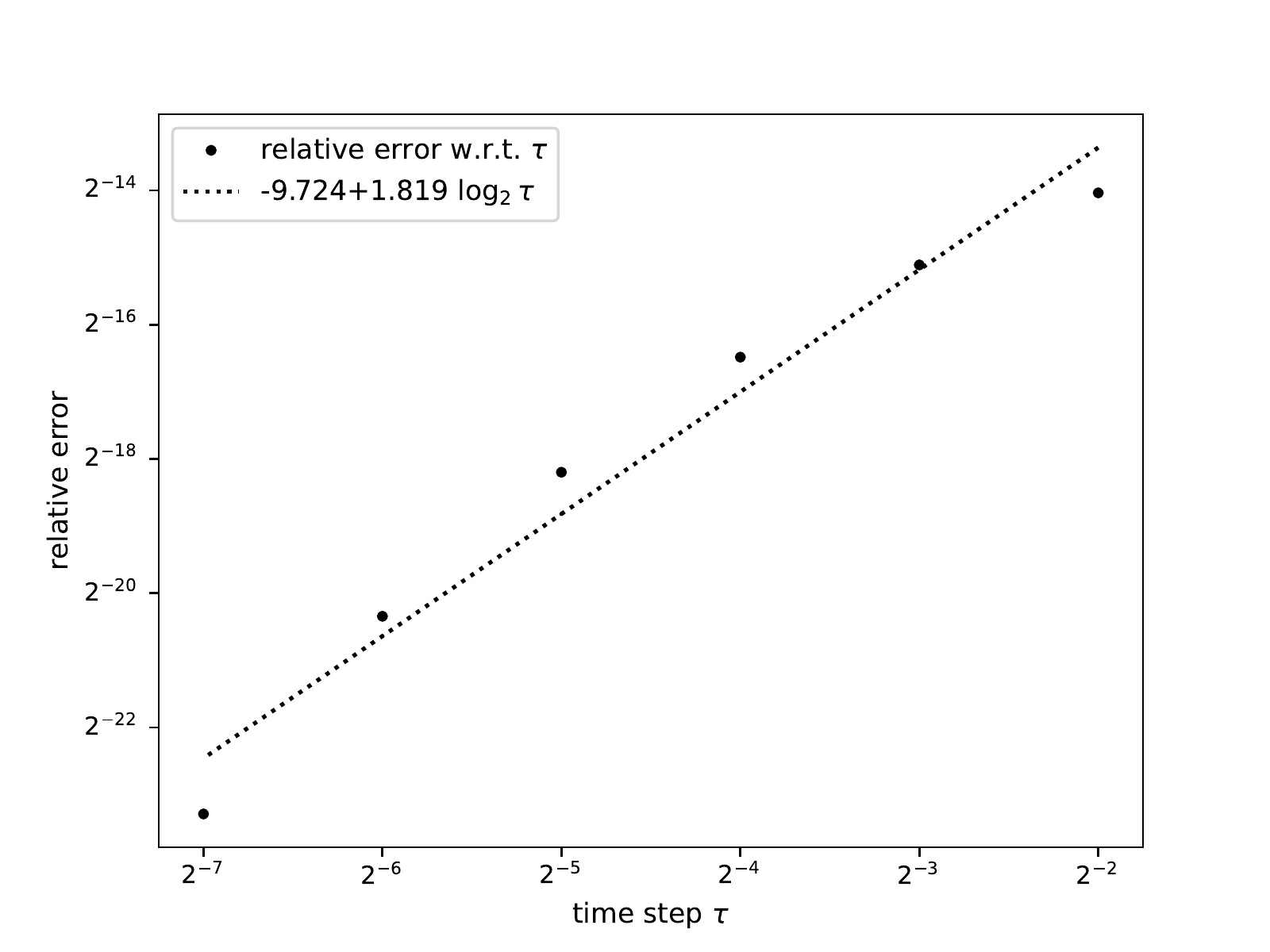}
	\end{subfigure}
	\caption{Relative error {$\tilde{E}_u$} w.r.t. the time step. \textit{Left:} the first experiment. \textit{Right:} the second experiment. Both numerical experiments show the potentially optimal convergence rate of the temporal discretization.}
	\label{fig:exp_convergence}
\end{figure}

\subsection{Numerical simulation}
\label{subsec:simulation}

We repeat the numerical simulation of an induction heating process performed in \cite{CGS2017} and consider a moving workpiece instead of a fixed one. The domain $\Om$ is a unit cube consisting of a thin-walled cylindrical aluminium workpiece with two radii $r_1 = 0.092 \mathrm{m}, r_2 = 0.081 \mathrm{m}$ and height $h = 0.3 \mathrm{m}$, a copper coil and surrounding air. The initial datum $\tilde{\Ab}(0) = \zrb$ in $\Theta(0)$ has a trivial extension $\tilde{\Ab}(0) = \zrb$ in the whole domain $\Om$ without requiring $\Cs^{2, 1}$ regularity of the boundary of $\Sm(0)$ and $\Pi$. Hence, our theoretical results are still valid for this geometry. The domain $\Om$ is partitioned into $32312$ tetrahedra. A static external current density with magnitude $\jmath = 1.0$E7$\mathrm{A/m^2}$ is driven through the coil $\Pi$ via the interfaces $\Gm_{\text{in}}$ and $\Gm_{\text{out}}$. The workpiece is moving along the $z$-axis with velocity $\vb = (0, 0, 0.46875)^\transpose \mathrm{cm/s}$, and the considered time length is $T = 32 \mathrm{s}$. Outside the workpiece, we assume that the velocity is very small; thus, the thermal convection is dominated by the thermal conduction. Therefore, we can neglect the thermal convection effect in the air domain and avoid the computation of airflow, which is not essential in this paper.

The reference solution is computed from the variational system \eqref{eq:TD_phi}-\eqref{eq:TD_u} with time step $\tau = 2^{-7}\mathrm{s}$ ($n = 4096$ subintervals). Different locations of the workpiece together with the corresponding temperature distributions in the conductors at two different time points are presented in \Cref{fig:simulation}. Some rougher discrete solutions are also computed when the number of time intervals equals $n = $ 64, 128, 256, 512, 1024 and 2048. Relative error $\tilde{E}_u$ on the temperature $u$ w.r.t. time step is shown in \Cref{fig:error}, which confirms the potentially optimal convergence rate of our proposed scheme.

\begin{figure}
	\centering
	\begin{subfigure}{.33\textwidth}
		\centering
		\includegraphics[width=0.95\textwidth]{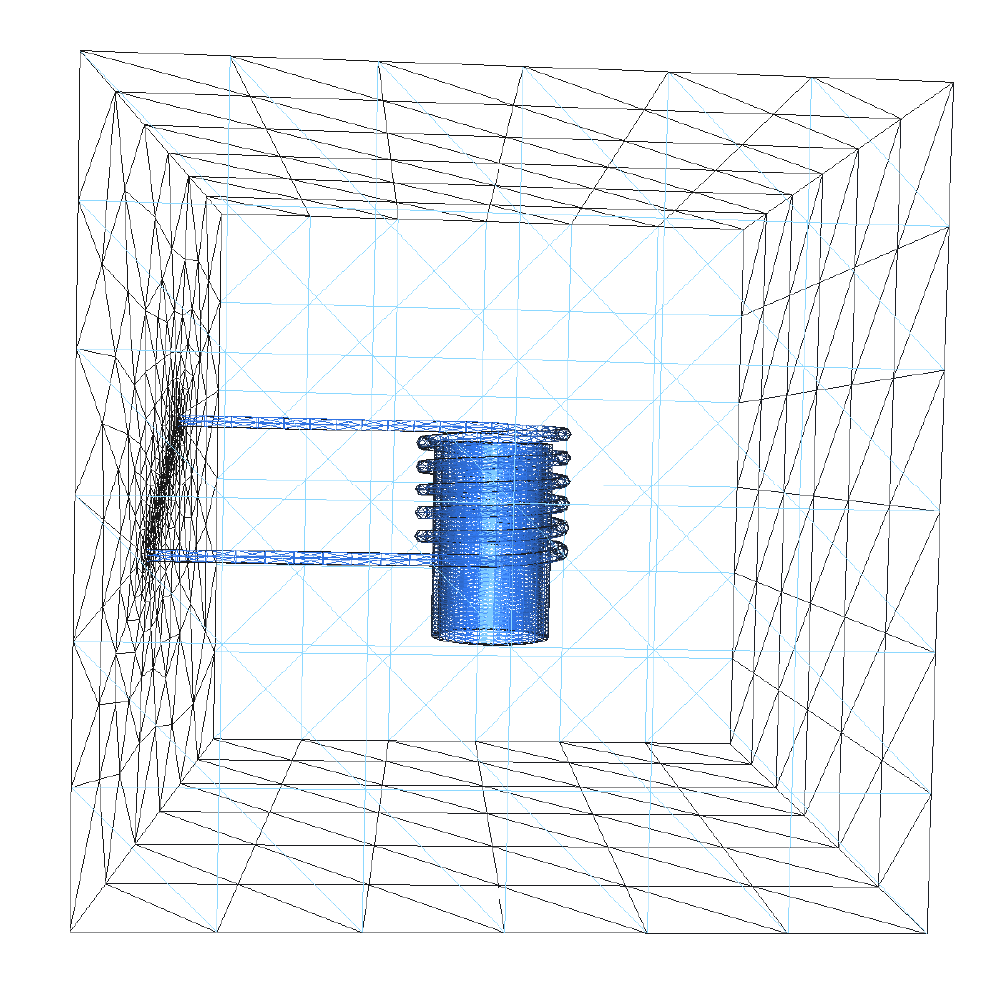}
	\end{subfigure}%
	\begin{subfigure}{.33\textwidth}
		\centering
		\includegraphics[width=0.9\textwidth]{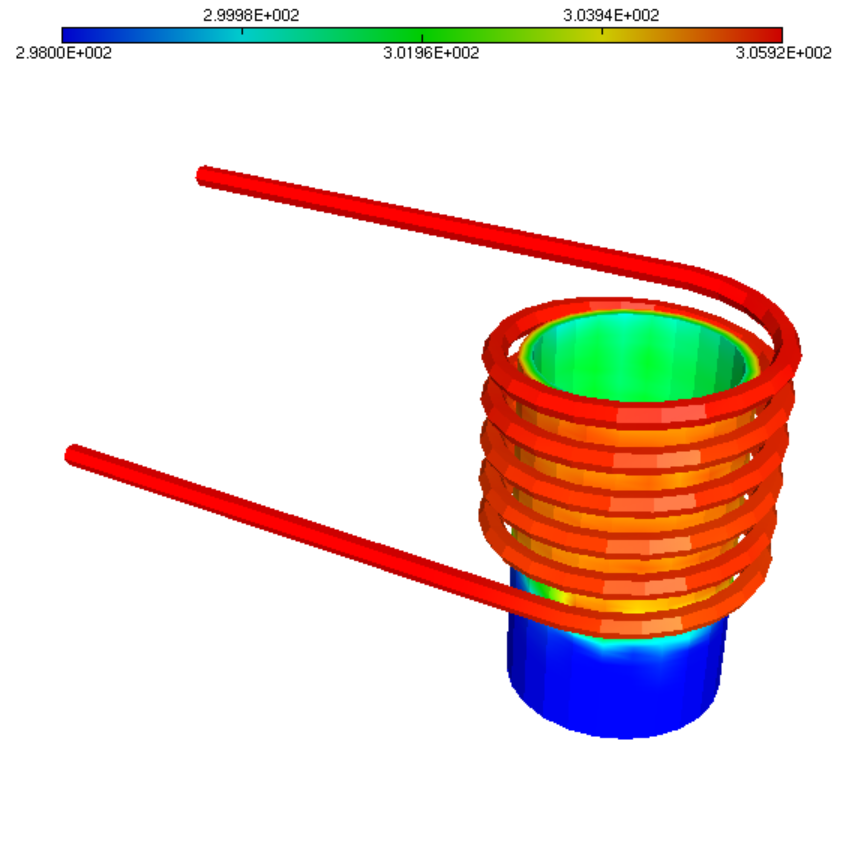}
	\end{subfigure}
	\begin{subfigure}{.33\textwidth}
		\centering
		\includegraphics[width=0.9\textwidth]{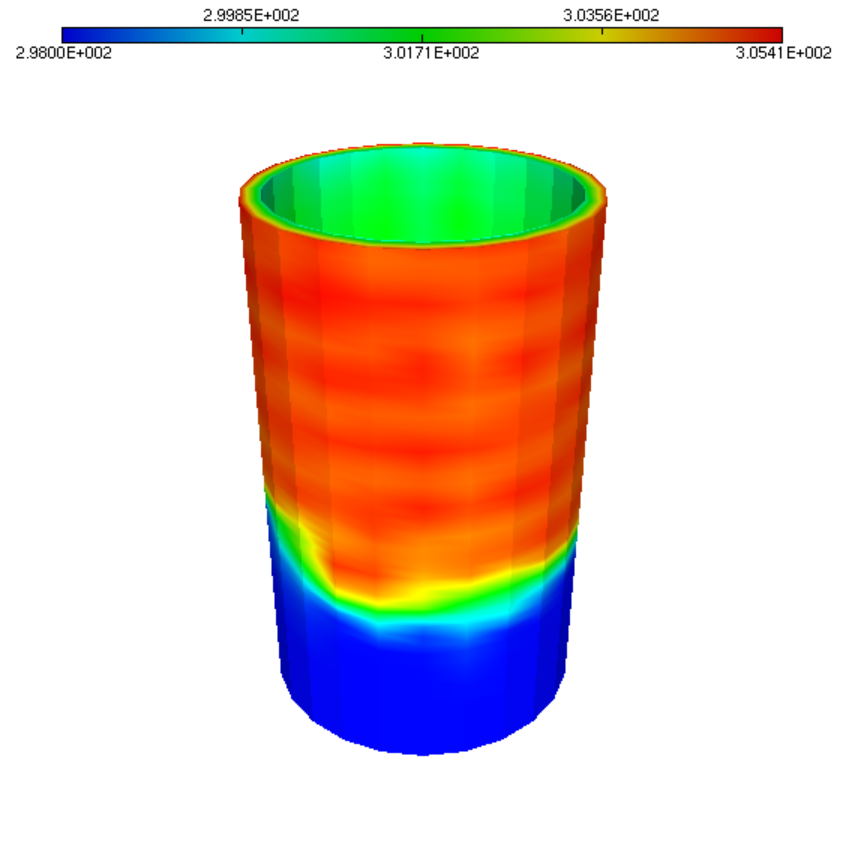}
	\end{subfigure}
	\\
	\begin{subfigure}{.33\textwidth}
		\centering
		\includegraphics[width=0.95\textwidth]{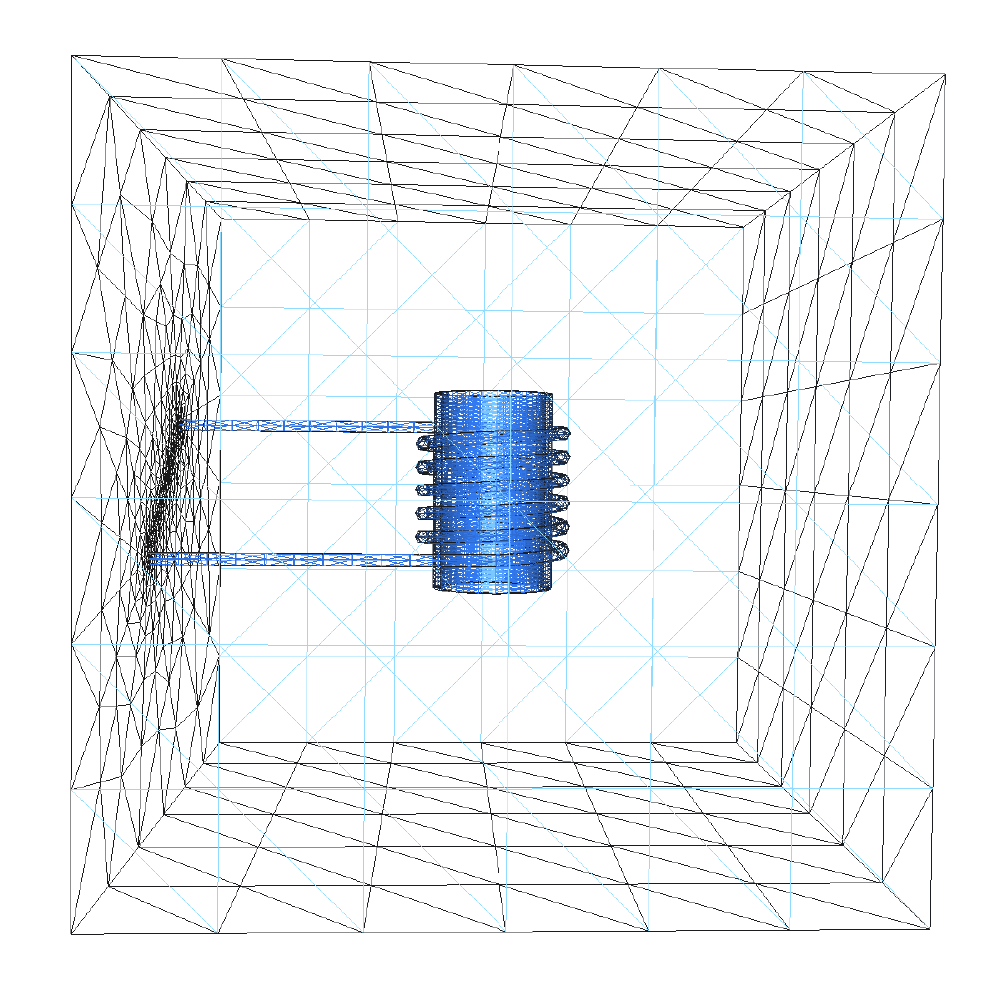}
	\end{subfigure}%
	\begin{subfigure}{.33\textwidth}
		\centering
		\includegraphics[width=0.9\textwidth]{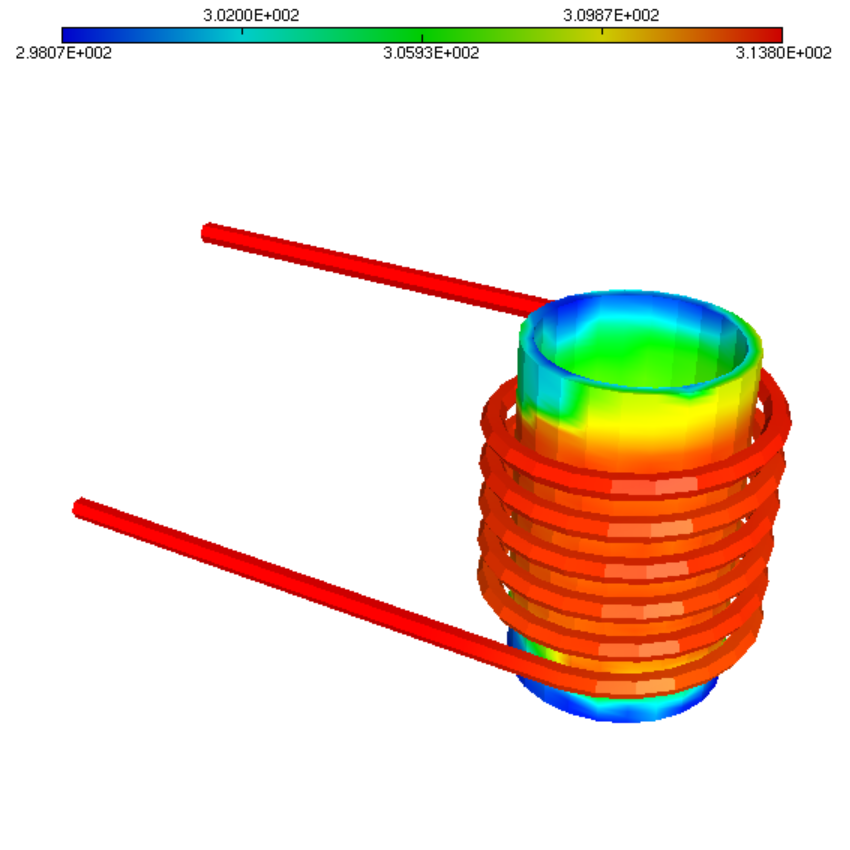}
	\end{subfigure}%
	\begin{subfigure}{.33\textwidth}
		\centering
		\includegraphics[width=0.9\textwidth]{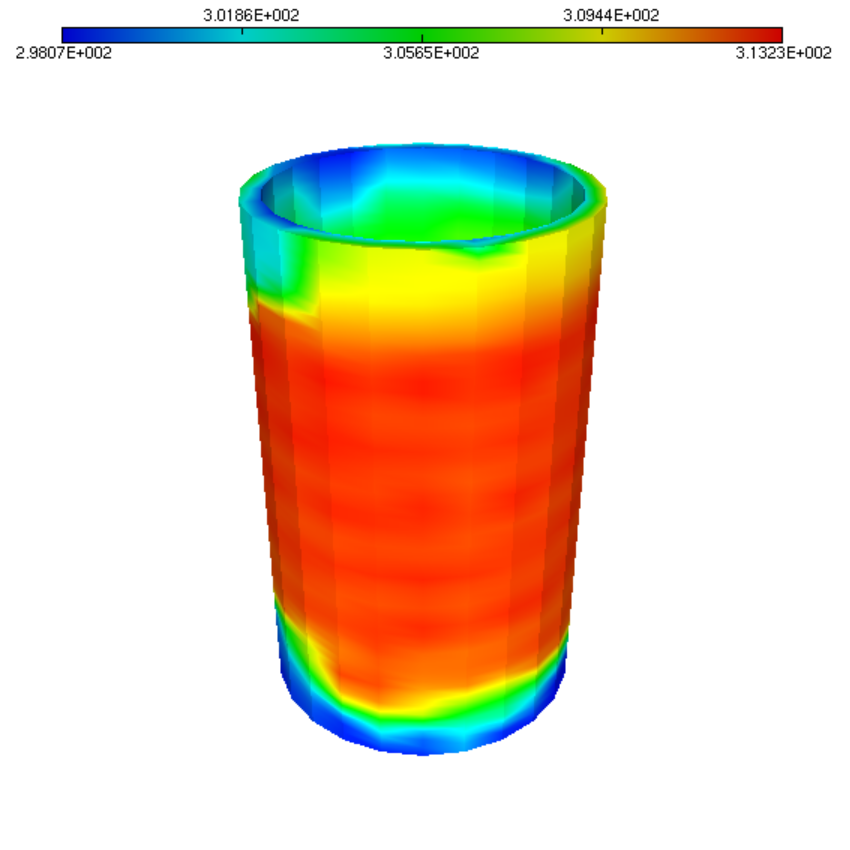}
	\end{subfigure}
	\caption{Different locations of the workpiece and the corresponding temperature distributions in the conductors at different time points. \textit{Top:} $t = 16 \mathrm{s}$. \textit{Bottom:} $t = 32 \mathrm{s}$.}
	\label{fig:simulation}
\end{figure}

\begin{figure}
	\centering
	\includegraphics[width=0.55\textwidth]{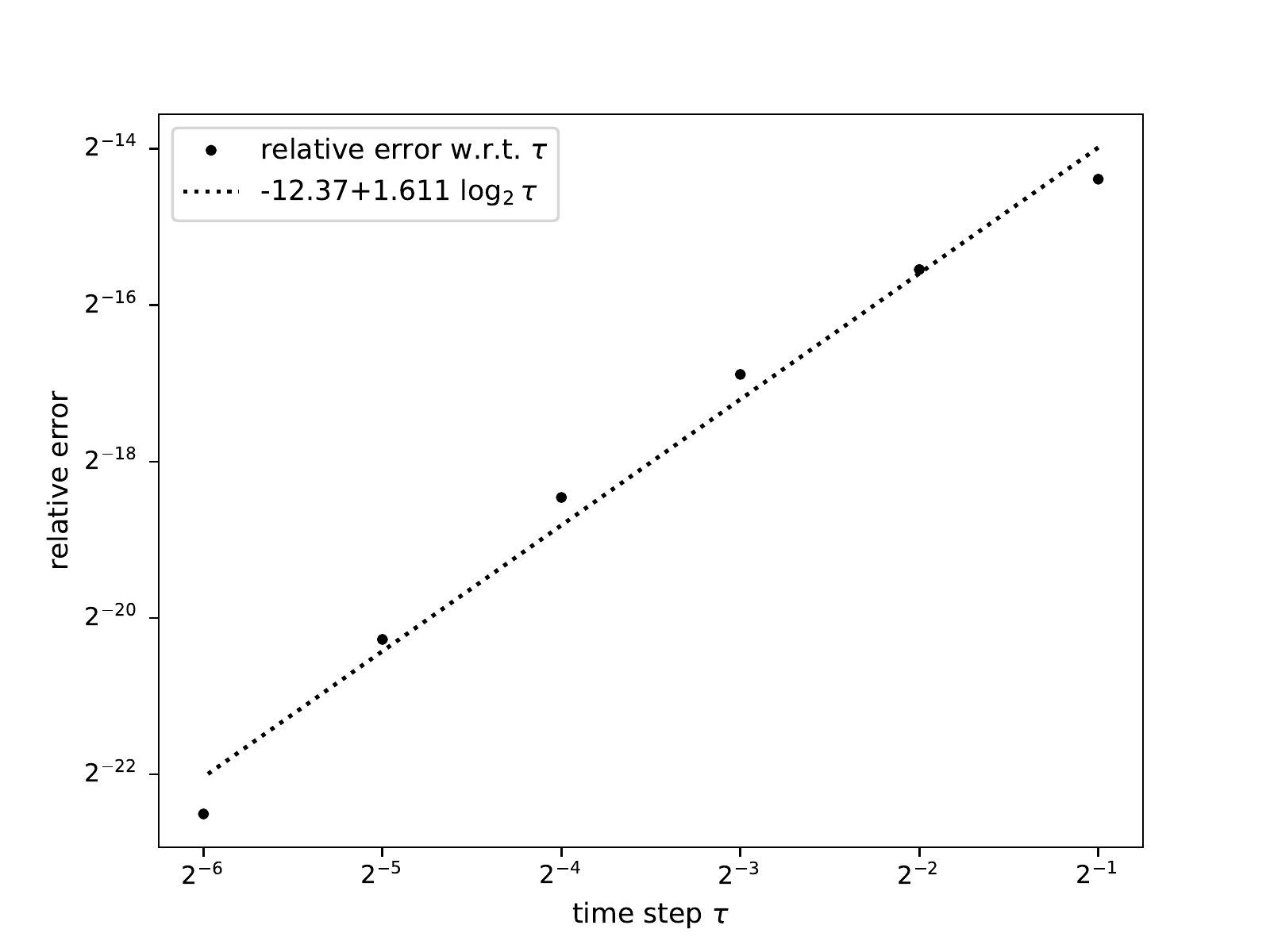}
	\caption{Relative error {$\tilde{E}_u$} w.r.t. time step. This result confirms the potentially optimal convergence rate of the numerical scheme.}
	\label{fig:error}
\end{figure}

\section{Conclusion}
\label{sec:conclusion}

In the present paper, we have investigated an induction heating problem in a three-dimensional domain containing a moving non-magnetic conductor. The electromagnetic process and heat transfer are modelled by PDEs, which are coupled by the Joule heating effect. A cut-off function has been introduced to restrain the nonlinear Joule heat source. A time-discrete scheme based on the backward Euler's method has been proposed to approximately solve the variational problems. The convergence of the proposed discretization scheme and the well-posedness of the variational system have been proved with the aid of the Reynolds transport theorem and Rothe's method. Some numerical results have been presented to support the theoretical results.

In the future, comprehensive error estimates of the proposed temporal discretization scheme should be performed to verify the potentially optimal convergence rate obtained numerically in \Cref{sec:numerial_results}. Future studies could also concern full-wave induction heating problems involving magnetic moving conductors, which have a wide range of applications in manufacturing industries. In addition, functional analysis performed in \Cref{sec:functional_setting} could be extended for more general settings, such as non-constant material

		\bibliography{ref}
		\bibliographystyle{abbrv}
\end{document}